\documentclass{article}
\usepackage{amsmath,amssymb,amsthm}
\usepackage{enumerate}
\usepackage[a4paper, total={5.5in, 9in}]{geometry}

\newtheorem{theorem}{Theorem}[section]
\newtheorem{lemma}[theorem]{Lemma}
\newtheorem{proposition}[theorem]{Proposition}
\newtheorem{corollary}[theorem]{Corollary}

\newtheorem{openProblem}[theorem]{Open problem}

\theoremstyle{definition}
\newtheorem{definition}[theorem]{Definition}
\newtheorem{remark}[theorem]{Remark}
\newtheorem{example}[theorem]{Example}

\newcommand{\id}{\operatorname{id}}

\newcommand{\St}{\operatorname{St}}

\newcommand{\hol}{\operatorname{Hol}}

\newcommand{\Z}{\mathbb{Z}}

\newcommand{\im}{\operatorname{Im}}
\renewcommand{\ker}{\operatorname{Ker}}

\newcommand{\sym}{\operatorname{Sym}}

\newcommand{\aut}{\operatorname{Aut}}
\newcommand{\ret}{\operatorname{Ret}}
\newcommand{\ord}[1]{\vert #1\vert}

\newcommand{\soc}{\operatorname{Soc}}
\newcommand{\core}{\operatorname{core}}

\begin{document}

\title{Solutions of the Yang-Baxter equation associated to skew left braces, with applications to racks
}

\author{David Bachiller\footnote{Research partially supported by a grant of MICIIN (Spain)
MTM2011-28992-C02-01.}}
\date{}
\maketitle

\begin{abstract}
 Given a skew left brace $B$, a method is given to construct all  the
non-degenerate set-theoretic solutions $(X,r)$ of the
Yang Baxter equation such that the  associated permutation
group $\mathcal{G}(X,r)$ is isomorphic, as a skew left brace,  to
$B$. This method depends entirely on the brace structure of $B$. We then adapt this method 
to show how to construct solutions with additional properties, like square-free, involutive or 
irretractable solutions. Using this result, it is even possible to recover racks from their permutation group. 
\end{abstract}



\section{Introduction}

The Yang-Baxter equation is a fundamental equation both in mathematical physics (see \cite{BaxterBook}), and in algebra, due to 
its relations with Hopf algebras and quantum groups (see \cite{Kassel}). 
Here by a solution of the Yang-Baxter equation we mean an invertible linear map $R: V\otimes V\to V\otimes V$, where 
$V$ is a vector space, such that 
$$(\id\otimes R)\circ (R\otimes \id)\circ (\id\otimes R)=(R\otimes \id)\circ (\id\otimes R)\circ (R\otimes \id)$$
over $V^{\otimes 3}$. It is an open problem to find all the solutions to this equation, and because of this, 
Drinfeld proposed in \cite{Drinfeld} to study the class of set-theoretic solutions 
as a more tractable case. A set-theoretic solution is a bijective map $r:X\times X\to X\times X$, 
where $X$ is a non-empty set, such that 
$$(\id\times r)\circ (r\times \id)\circ (\id\times r)=(r\times \id)\circ (\id\times r)\circ (r\times \id)$$
over $X^3$; 
by linear extension, a set-theoretic solution always induces a usual solution. This class of solutions has received a 
lot of attention recently due to the combinatorial and algebraic structures associated to them. 

The first solutions of this type to be studied were the non-degenerate involutive solutions. Initially, in 
\cite{ESS,GVdB} they were studied using two associated groups, the structure group $G(X,r)$ and the 
permutation group $\mathcal{G}(X,r)$ of a solution $(X,r)$. Then, Rump in \cite{Rump1} adapted the results in \cite{ESS,GVdB} to a more natural setting
by introducing a new algebraic structure called left brace.
A left brace is a set $B$ with two operations $+$ and $\cdot$ such that $(B,+)$ is an abelian group, $(B,\cdot)$ is a group, and 
any $a,b,c\in B$ satisfy $a\cdot (b+c)+a=a\cdot b+a\cdot c$. 
Rump showed that both $G(X,r)$ and $\mathcal{G}(X,r)$ are left braces, and that many of the previous results are more naturally 
stated in terms of this structure. A good introduction to left braces can be found in \cite{CJOarXiv}. 

This new approach turned out to be very fruitful \cite{B2,B3,BCJ,BCJO,CJO,CGS,GatBraces,Agata2,Agata}. 
In particular, in \cite{BCJ} a method was presented to 
construct, from a given left brace $B$, all the non-degenerate involutive solutions $(X,r)$ such that $\mathcal{G}(X,r)\cong B$ as braces, 
reducing the classification of all the non-degenerate involutive solutions
to the classification of all the left braces.

Recently, an analogous approach for non-involutive solutions has been proposed. These solutions were first studied in
 \cite{Soloviev, LYZ} with a combination of algebraic and combinatorial techniques, using three associated groups to a solution $(X,r)$: 
the structure group $G(X,r)$, the derived group $A(X,r)$, and the permutation group $\mathcal{G}(X,r)$.  Then, Guarnieri and Vendramin in \cite{GV} 
adapted these results using a generalization of left braces called skew left braces. A skew left brace is a
set $B$ with two operations $\star$ and $\cdot$ such that $(B,\star)$ and $(B,\cdot)$ are groups, and any $a,b,c\in B$ satisfy 
$a\cdot(b\star c)=(a\cdot b)\star a^\star \star (a\cdot c)$, where $a^\star$ denotes the inverse of $a$ in $(B,\star)$. Note that left braces are precisely skew left braces with $(B,\star)$ abelian. 
It is shown in \cite{GV} that $G(X,r)$ and $\mathcal{G}(X,r)$ are skew left braces, and that $A(X,r)$ is isomorphic to the group $(G(X,r),\star)$. 
Again, the results of \cite{Soloviev, LYZ} turn out to be more natural in this setting. 

The main aim of this paper is to exploit this generalization to obtain analogous results to the ones in \cite{BCJ}: we show how to construct, 
from a given skew left brace $B$, all the non-degenerate solutions $(X,r)$ such that $\mathcal{G}(X,r)\cong B$ as braces. 
Moreover, isomorphisms of solutions correspond to certain automorphisms of $B$. Hence, this translates 
the construction of non-degenerate solutions to a purely brace-theoretical problem, and, in particular, 
reduces the classfication of non-degenerate solutions to the classification of skew left braces. To prove these 
results, a important tool is the result \cite[Theorem~2.7]{Soloviev}.

We start in Section~\ref{sectSkewBraces} by introducing some concepts about skew left braces. 
In particular, we present some new group actions induced by any skew left braces that 
will be fundamental in the next sections.
Although some of the results are originally from \cite{GV}, we restate them in a convenient way for 
our purposes.  After that, in Section~\ref{sectNonDegSol}, we first recall the basic notions 
about non-degenerate solutions, and their connection with skew left braces. With that at hand, we present our main results about 
construction and isomorphism of non-degenerate solutions 
in Section~\ref{sectConstSol}.

Then, we adapt our method to construct non-degenerate solutions with additional properties, like square-free (Section~\ref{SquareFree}), involutive (Section~\ref{constInvSol}), or irretractable (Section~\ref{constIrret}) solutions. 
The result about the involutive case is originally from \cite{BCJ}, but here we show how it fits in the general setting. 
As an example of all these techniques, we show another way to construct the example of a square-free irretractable solution of size $8$ 
due to Vendramin \cite{VendraminCounter}. 
Next, we devote Section~\ref{sectRacks} to racks and quandles, which are important combinatorial invariants of knots (see \cite{Nelson}). 
We show that they correspond to the case of trivial skew left braces, and using our main result, 
we describe how to recover any rack from its permutation group. 

Finally, two parts are somehow disconnected from the rest: group-theoretical results about $G(X,r)$ and $A(X,r)$ in Section~\ref{sectGroupTheory}, 
and Hopf-Galois extensions in Section~\ref{sectHopfGalois}. 
The former is included because both $G(X,r)$ and $A(X,r)$ are important groups in our exposition, and it is important to understand 
better their structure. The later is introduced because we think that this connection will be useful in the future, and because
some of the best results about classification of skew left braces, our ultimate goal, are done in this setting.

\section{Skew left braces: Initial concepts and results}\label{sectSkewBraces}

Motivated by the work of \cite{Soloviev,LYZ} (which we will recall in the next section), 
Guarnieri and Vendramin introduced in \cite{GV} the concept of skew left braces, in order to generalize 
left braces (introduced by Rump in \cite{Rump1}) and their applications to solutions of the Yang-Baxter equation. 

\begin{definition}
A {\em skew left brace} is a set $B$ with two binary operations, 
$\star$ and $\cdot$, such that $(B,\star)$ and $(B,\cdot)$ are groups, 
and any $a,b,c\in B$ satisfy
$$
a\cdot(b\star c)=(a\cdot b)\star a^\star \star(a\cdot c),
$$
where $a^\star$ denotes the inverse of $a$ for the operation $\star$. 
A {\em skew right brace} is defined analogously, changing 
the last property by $(b\star c)\cdot a=(b\cdot a)\star a^\star\star (c\cdot a)$. 
A skew left brace also satisfying the condition of a 
skew right brace is called a {\em skew two-sided brace}. 
\end{definition}

We denote the inverse of $a$ for the operation $\cdot$ by $a^{-1}$. We call $(B,\star)$ the star 
group of $B$, and $(B,\cdot)$, the multiplicative group
 of $B$.

Observe that the identity element of the two operations coincide: 
$a\cdot 1_\star=a\cdot (1_\star \star 1_\star)=(a\cdot 1_\star)\star a^\star \star (a\cdot 1_\star),$
and multiplying on the left by $[(a\cdot 1_\star)\star a^\star]^\star$ with respect to the $\star$ 
operation, we get $a\cdot 1_\star=a$ for any $a\in B$.
We denote this common identity element by $1(=1_\cdot=1_\star)$.

\begin{example}\label{FirstExNLB}
\begin{enumerate}[(a)]
\item Any group $(G,\cdot)$ with second operation given by $g\star h:=g\cdot h$ is a 
skew two-sided brace. This structure is called a {\em trivial brace}. 

\item Any group $(G,\cdot)$ with star product given by $g\star h:=h\cdot g$ is 
a skew two-sided brace. If $(G,\cdot)$ is abelian, this example coincides with the 
previous one. 

\item Take $\Z/(2n)=\langle \gamma\rangle$ the cyclic group of $2n$ elements written multiplicatively. 
Define the star product $\gamma^a\star\gamma^b:=\gamma^{(-1)^b a+b}$, for any $\gamma^a,\gamma^b\in \Z/(2n)$. 
Then $(\Z/(2n),\star,\cdot)$ is a skew two-sided brace. 

\item Take the group $G=\langle \sigma, \tau: \sigma^7=\tau^3=1,\tau\star\sigma=\sigma^2\star\tau\rangle$ 
(which is the non-abelian group of order $3\cdot 7$)
as star group, and define the product $(\sigma^a\star \tau^b)\cdot (\sigma^i\star \tau^j):=\sigma^{a+4^bi}\star\tau^{b+j}$, 
for any $\sigma^a\star \tau^b$, $\sigma^i\star\tau^j\in G$. 
Then, $(G,\star,\cdot)$ is a skew left brace. Note that 
$(\tau\star \tau)\cdot \sigma=\tau^2\cdot\sigma=\sigma^{16}\star \tau^2=\sigma^{2}\star\tau^2,$
but
$(\tau\cdot\sigma)\star\sigma^\star\star (\tau\cdot \sigma)=\sigma^4\star\tau\star\sigma^\star\star\sigma^4\star\tau=
\sigma^3\star\tau^2$. So this is an example of a skew left brace which is not two-sided. 
\end{enumerate}
\end{example}

Some of the important properties of skew left braces are related to 
actions of this algebraic structure over different sets. In the next three lemmas, 
we define the actions $\lambda$, $\gamma$ and $\Theta$, which are fundamental for the 
algebraic study of skew left braces, and prove their main properties. 

\begin{lemma}[{\cite[Corollary~1.10]{GV}}]\label{propieLambda}
Let $B$ be a skew left brace. For every $a\in B$, define
a map $\lambda_a: B\to B$ by  $\lambda_a(b):=a^\star\star(a\cdot b),$ for any
$b\in B$. Then
\begin{enumerate}[(i)]
\item $\lambda_a(b\star c)=\lambda_a(b)\star \lambda_a(c);$
\item $\lambda_{a\cdot b}=\lambda_a\circ \lambda_{b};$
\item $\lambda_a$ is bijective, with inverse equal to $\lambda_{a^{-1}}$;
\item the map $\lambda: (B,\cdot)\to\aut(B,\star)$, $a\mapsto \lambda_a$, is a morphism of groups.
\end{enumerate}
\end{lemma}

\begin{lemma}\label{propieg}
Let $B$ be a skew left brace. For every $b\in B$, define
a map $\gamma_b: B\to B$ by  $\gamma_b(a):=\lambda^{-1}_{\lambda_a(b)}((a\cdot b)^{\star}\star a\star (a\cdot b)),$ for any
$a\in B$. Then, for any $a,b\in B$,
\begin{enumerate}[(i)]
\item $\gamma_b(a)=\lambda^{-1}_{\lambda_a(b)}(\lambda_a(b)^\star\star a\star \lambda_a(b))
=(b^{-1}\cdot (a^{-1}\star b))^{-1}=((b^{-1}\cdot a^{-1})\star (b^{-1})^{\star})^{-1}$;
\item $\gamma_a$ is bijective, with inverse equal to $\gamma_{a^{-1}}$;
\item $\gamma_a\circ \gamma_b=\gamma_{b\cdot a}$.
\end{enumerate}
\end{lemma}
\begin{proof}
\begin{enumerate}[(i)]
\item The first equality is due to some cancellations of the star operation: 
\begin{align*}
\lambda^{-1}_{\lambda_a(b)}(\lambda_a(b)^\star\star a\star \lambda_a(b))&=
\lambda^{-1}_{\lambda_a(b)}((a\cdot b)^\star\star a\star a\star a^\star\star (a\cdot b))\\
&=\lambda^{-1}_{\lambda_a(b)}((a\cdot b)^\star\star a\star (a\cdot b))=\gamma_b(a).
\end{align*}

With this first equality, we rewrite $\gamma_b$ in terms of 
the operations of $B$:
\begin{align*}
\gamma_b(a)&=\lambda^{-1}_{\lambda_a(b)}(\lambda_a(b)^\star\star a\star \lambda_a(b))
=\lambda_a(b)^{-1}\cdot \left(\lambda_a(b)\star \lambda_a(b)^\star\star a\star \lambda_a(b)\right)\\
&=\lambda_a(b)^{-1}\cdot (a\star \lambda_a(b))
=\lambda_a(b)^{-1}\cdot a\cdot b=
(b^{-1}\cdot a^{-1}\cdot \lambda_a(b))^{-1}\\
&=(b^{-1}\cdot (a^{-1}\star b))^{-1}=((b^{-1}\cdot a^{-1})\star (b^{-1})^{\star})^{-1}.
\end{align*}

\item With this last expression for $\gamma$, we see that $\gamma_b$ is invertible, 
with inverse $$a\mapsto (b\cdot (a^{-1}\star b^{-1}))^{-1}.$$

\item Now we check that $\gamma:(B,\cdot)\to\sym_B$ is an anti-morphism. For 
this, we use that $\gamma_b(a)=((b^{-1}\cdot a^{-1})\star (b^{-1})^{\star})^{-1}$, as we have seen before. 
On one side, for any $a,b,c\in B$, 
\begin{align*}
\gamma_{ba}(c)&=\left[(a^{-1}\cdot b^{-1}\cdot c^{-1})\star (a^{-1}\cdot b^{-1})^\star\right]^{-1}\\
&=
\left[(a^{-1}\cdot b^{-1}\cdot c^{-1})\star (a^{-1}\star \lambda_{a^{-1}}( b^{-1}))^\star\right]^{-1}\\
&=\left[(a^{-1}\cdot b^{-1}\cdot c^{-1})\star \lambda_{a^{-1}}( b^{-1})^\star \star (a^{-1})^\star\right]^{-1}\\
&=\left[(a^{-1}\cdot b^{-1}\cdot c^{-1})\star \lambda_{a^{-1}}( (b^{-1})^\star) \star (a^{-1})^\star\right]^{-1}\\
&=\left[(a^{-1}\cdot b^{-1}\cdot c^{-1})\star (a^{-1})^\star\star(a^{-1}\cdot (b^{-1})^\star) \star (a^{-1})^\star\right]^{-1},
\end{align*}
and, on the other, 
\begin{align*}
\gamma_a(\gamma_b(c))&=\gamma_a(\left[ (b^{-1}\cdot c^{-1})\star (b^{-1})^\star\right]^{-1})\\
&=
\left[ (a^{-1}\cdot \left[(b^{-1}\cdot c^{-1})\star (b^{-1})^\star\right])\star (a^{-1})^\star\right]^{-1}\\
&=\left[ (a^{-1}\cdot b^{-1}\cdot c^{-1})\star (a^{-1})^\star\star (a^{-1}\cdot (b^{-1})^\star)\star (a^{-1})^\star\right]^{-1}.
\end{align*}

\end{enumerate}
\end{proof}

\begin{lemma}\label{propieAction}
Let $B$ be a skew left brace. 
Let $(B,\star)\rtimes (B,\cdot)$ be the semidirect product of groups defined by the action 
$\lambda: (B,\cdot)\to\aut(B,\star)$ of Lemma~\ref{propieLambda}. 
Then, the map
$$
\begin{array}{cccc}
\Theta : & (B,\star)\rtimes (B,\cdot) & \longrightarrow & \aut(B,\star)\\
 & (a,b) & \mapsto & \Theta_{(a,b)} : [c\mapsto a\star \lambda_b(c)\star a^\star]
\end{array}
$$
is a morphism of groups. 
\end{lemma}
\begin{proof}
It is direct to check that $\Theta_{(a,b)}$ is an automorphism of $(B,\star)$ for any $a,b\in B$, so 
$\Theta$ is well-defined. 
Now, the following argument proves that $\Theta$ is a morphism: for any $a_1,b_1,a_2,b_2,c\in B$,
\begin{align*}
\Theta_{(a_1,b_1)(a_2,b_2)}(c)&=\Theta_{(a_1\star \lambda_{b_1}(a_2),b_1\cdot b_2)}(c)
=(a_1\star \lambda_{b_1}(a_2))\star \lambda_{b_1\cdot b_2}(c)\star (a_1\star \lambda_{b_1}(a_2))^\star\\
&=a_1\star \lambda_{b_1}(a_2)\star \lambda_{b_1\cdot b_2}(c)\star \lambda_{b_1}(a_2)^\star\star a_1^\star
=a_1\star \lambda_{b_1}(a_2\star \lambda_{b_2}(c)\star a_2^\star)\star a_1^\star\\
&=\Theta_{(a_1,b_1)}\Theta_{(a_2,b_2)}(c).
\end{align*} 
\end{proof}

As in any algebraic structure, it is interesting to study substructures, and 
to be able to define a quotient structure with respect to some of this 
substructures. 

\begin{definition}
Let $B$ be a skew left brace. A subset $I$ of $B$ is an {\em ideal} of $B$ if 
it is a normal subgroup of $(B,\star)$, a normal subgroup of 
$(B,\cdot)$, and $\lambda_b(y)\in I$ for any $y\in I$, $b\in B$.

A {\em sub-brace} is a subset of $B$ which is a subgroup of $(B,\cdot)$, and a subgroup of $(B,\star)$. 
A {\em left ideal} is a subgroup $L$ of $(B,\star)$ such that $\lambda_b(y)\in L$, for 
any $b\in B$ and $y\in L$.
\end{definition}

Note that a left ideal $L$ of $B$ is always a sub-brace of $B$, since $a\cdot b=a\star\lambda_a(b)\in L$ 
and $a^{-1}=\lambda_{a^{-1}}(a)^\star\in L$, for any $a,b\in L$. By \cite[Lemma~2.3]{GV},
if $I$ is an ideal of $B$, we can define a structure of skew left brace over the quotient $B/I$.
We can also define in the standard way morphisms of skew left braces
and, as usual, ideals are precisely the kernels of the possible morphisms.

One ideal that is particularly interesting is the socle. 

\begin{definition}
Let $B$ be a skew left brace. We define the {\em socle of $B$} as 
$$
\soc(B):=\{b\in B~:~\lambda_b=\id,\text{ and } a\star b\star a^{\star}=b\text{ for any } a\in B\}.
$$
\end{definition}

The socle appears as the kernel of some of the actions defined before. 

\begin{proposition}[{\cite[Lemma~2.5]{GV}}]\label{SocleKernel}
\begin{enumerate}[(i)]
\item The socle is an ideal of $B$, contained in the center of $(B,\star)$. 
For any $a\in B$ and any $y\in\soc(B)$, $\lambda_a(y)=a\cdot y\cdot a^{-1}$.
Moreover, $\soc(B)$ is a trivial brace, so $(\soc(B),\cdot)$ is an abelian group. 
\item The socle coincides with the kernel of the morphism of groups
$$
\begin{array}{cccc}
\varphi: &(B,\cdot)& \to &\aut(B,\star)\times\sym_B\\
 & a & \mapsto & (\lambda_a,\gamma_a^{-1}).
\end{array}
$$

It is also the kernel of the morphism of groups
$$
\begin{array}{cccc}
\Phi:& (B,\cdot)& \to &\aut(B,\star)\times\aut(B,\star)\\
& a & \mapsto & (\lambda_a,h_a),
\end{array}
$$
where $h_a(b)=a\star \lambda_a(b)\star a^\star$.
\end{enumerate}
\end{proposition}
\begin{proof}
\begin{enumerate}[(i)]
\item Let $y,y'\in\soc(B)$ and $a,b\in B$. Note that $\lambda_{y\cdot y'}=\lambda_y\lambda_{y'}=\id$, and 
$a\star (y\cdot y')\star a^{\star}=a\star y\star y'\star a^{\star}=y\star y'=y\cdot y'$, so $y\cdot y'\in\soc(B)$. We also have that 
$\lambda_{y^{-1}}=\lambda^{-1}_y=\id$ and that $a\star y^{-1}\star a^\star=a\star \lambda_{y^{-1}}(y)^\star\star a^\star=
a\star y^\star\star a^\star=y^\star=y^{-1}$, 
 so $y^{-1}\in\soc(B)$, and then $\soc(B)$ is a subgroup of $(B,\cdot)$. Moreover, 
$\lambda_{b\cdot y\cdot b^{-1}}=\lambda_b\lambda_y\lambda_{b^{-1}}=\lambda_b\circ\id\circ\lambda_{b^{-1}}=\id$, 
and 
$a\star (b\cdot y\cdot b^{-1})\star a^\star=b\cdot y\cdot b^{-1}$ because
$b^{-1}\cdot (a\star (b\cdot y\cdot b^{-1})\star a^\star)=
(b^{-1}a)\star (b^{-1})^\star\star (yb^{-1})\star (b^{-1})^\star\star(b^{-1}a^\star)=
(b^{-1}a)\star (b^{-1})^\star\star y \star(b^{-1}a^\star)=
y\star (b^{-1}a)\star (b^{-1})^\star\star (b^{-1}a^\star)=
y\star b^{-1}(a\star a^{\star})=y\star b^{-1}=y\cdot b^{-1}$, 
using in one of the steps that 
$(y\cdot b)\star b^\star=(y\star b)\star b^\star=y$ when $y$ is in the socle. So it is a normal 
subgroup of $(B,\cdot)$.

Besides, $\lambda_b(y)=b^\star \star (byb^{-1})\star b=byb^{-1}$, 
thus $\soc(B)$ is closed by the lambda maps. Closed by lambda maps and closed by $\cdot$ implies 
that it is closed by $\star$. And it is a normal subgroup of $(B,\star)$ by definition (note that 
in fact it is contained in the center of $(B,\star)$ by definition).

Finally, in $\soc(B)$, $\cdot$ and $\star$ coincide: 
for any $x,y\in B$, $x\cdot y=x\star\lambda_x(y)=x\star \id(y)=x\star y$. In other words, 
$\soc(B)$ is a trivial brace. 

\item By Lemmas~\ref{propieLambda} and \ref{propieg}, the map $\varphi$ is well-defined, and it is a morphism. 
Now recall that $\gamma_a(b)=((a^{-1}\cdot b^{-1})\star (a^{-1})^{\star})^{-1}$, so if $\lambda_a=\id$, then 
$\gamma_a(b)=((a^{-1}\cdot b^{-1})\star (a^{-1})^{\star})^{-1}=((a^{-1}\star b^{-1})\star (a^{-1})^{\star})^{-1}=
(a^{-1}\star b^{-1}\star (a^{-1})^{\star})^{-1},$ thus $\gamma_a(b)=b$ if and only if $a^{-1}\star b^{-1}=b^{-1}\star a^{-1}$. Hence
\begin{align*}
\ker(\varphi)&=\{a\in B:\lambda_a=\gamma_a=\id\}\\
&=\{a\in B:\lambda_a=\id,~b\star a=a\star b\text{ for any }b\in B\}\\
&=\soc(B).
\end{align*}

The part about $\Phi$ is analogous, using Lemma~\ref{propieAction}.
\end{enumerate}
\end{proof}

These are all the initial concepts and results about skew left braces 
that we are going to need in this paper.

\section{Non-degenerate set-theoretic solutions}\label{sectNonDegSol}

\subsection{Connections with skew left braces}
Our aim in this section is to recall the class of non-degenerate set-theoretic solutions of the Yang-Baxter equation, 
and to show that skew left braces are an adequate algebraic structure to study them. 
Some of the most fundamental results of this section are due to Soloviev \cite{Soloviev} and Lu, Yan and Zhu \cite{LYZ}.
We also follow the exposition of \cite{Takeuchi} and \cite{GV}.

\begin{definition}
A {\em set-theoretic solution of the Yang-Baxter equation} is a pair $(X,r)$, where 
$X\neq \emptyset$ is a set, and $r$ is a bijective map $r:X\times X\to X\times X$ such that 
$r_1\circ r_2\circ r_1=r_2\circ r_1\circ r_2$, where
$r_1=r\times \id$ and $r_2=\id\times r$.
\end{definition}

\paragraph{Notation: } From now on, by a {\em solution} we will mean a set-theoretic solution of the Yang-Baxter 
equation. Besides, the components of any map $r:X\times X\to X\times X$ will be denoted by
$r(x,y)=(f_x(y),g_y(x))$, for any $x,y\in X$.

\paragraph{}
With this notation, the conditions to be a solution can be written more explicitly as follows:
$(X,r)$ is a solution if and only if $f$ and $g$ satisfy the three equalities
\begin{align}
f_{f_x(y)}f_{g_y(x)}(z)&=f_x f_y(z),\label{SolCond1}\\ 
g_{f_{g_y(x)}(z)}f_x(y)&=f_{g_{f_y(z)}(x)}g_z(y),\text{ and }\label{SolCond2}\\ 
g_zg_y(x)&=g_{g_z(y)}g_{f_y(z)}(x), \label{SolCond3}
\end{align}
for any $x,y,z\in X$.

\begin{definition}
Given two solutions $(X,r)$ and $(Y,s)$, a map $F:X\to Y$ is a {\em morphism of solutions} of the Yang-Baxter equation if 
it satisfies $s(F(x_1),F(x_2))=((F\times F)\circ r)(x_1,x_2)$ for any $x_1,x_2\in X$. If $F$ is bijective, we say that $F$ is 
an isomorphism of solutions. 
\end{definition}

If $r(x_1,x_2)=(f_{x_1}(x_2),g_{x_2}(x_1))$ and $s(y_1,y_2)=(f'_{y_1}(y_2), g'_{y_2}(y_1))$, then 
$F$ is a morphism of solutions if and only if $F(f_{x_1}(x_2))=f'_{F(x_1)}F(x_2)$ and $F(g_{x_2}(x_1))=g'_{F(x_2)}F(x_1)$ 
for any $x_1,x_2\in X$. 

\paragraph{} Since the initial work of Etingof, Schedler and Soloviev \cite{ESS}, and Gateva-Ivanova and Van den Bergh \cite{GVdB}, 
set-theoretic solutions are not studied in general, but adding some additional properties that allow 
us to study them with techniques from group theory. The most usual ones are non-degeneracy, and involutivity.

\begin{definition}
We say that a solution $(X,r)$ is {\em non-degenerate} if $f_x$ and $g_x$ are invertible for any $x\in X$. 

We say that $(X,r)$ is {\em involutive} if $r^2=r\circ r=\id$; in other words, since
\begin{align*}
r^2(x,y)=r(f_x(y),g_y(x))=(f_{f_x(y)}g_y(x),~g_{g_y(x)}f_x(y)),
\end{align*}
a solution is involutive if and only if $f_{f_x(y)}g_y(x)=x$ and $g_{g_y(x)}f_x(y)=y$ for any
$x,y\in X$.
\end{definition}

From now on, we focus on non-degenerate solutions. In this section, we study non-degenerate solutions in general, and 
then we will study non-degenerate involutive solutions as a particular case.

\paragraph{} As we have said, we can study non-degenerate solutions using group theory. 
In this strategy, the following two groups are fundamental.

\begin{definition}
Let $(X,r)$ be a non-degenerate solution. The structure group of $(X,r)$ is defined by
the presentation 
$$
(G(X,r),\cdot):=\left\langle X\mid x\cdot y=f_x(y)\cdot g_y(x)\text{, for } x,y\in X\right\rangle.
$$
The derived group of $(X,r)$ is defined by the presentation
$$
(A(X,r),\star):=\left\langle X\mid x\star f_x(y)=f_x(y)\star f_{f_x(y)}g_y(x)\text{, for } x,y\in X\right\rangle.
$$
\end{definition}

We denote by $i_G: X\to G(X,r)$ and $i_A: X\to A(X,r)$ the two natural maps. As the following example 
shows, these maps are not always injective.

\begin{example}(due to P. Etingof, see \cite[page 17]{LYZ})
Let $X$ be a set, and let $f,g: X\to X$ be two bijective maps. Define $r: X\times X\to X\times X$
by $r(x,y)=(f(y),g(x))$, $x,y\in X$. Then one can check that $r$ is a solution if and only if 
$f\circ g=g\circ f$. Inside $G(X,r)$, we have $x\cdot y=f(y)\cdot g(x)=f(g(x))\cdot g(f(y))=
f(g(x))\cdot f(g(y)).$ If $x$ is a fixed point of $f\circ g$, then $y=f(g(y))$ in $G(X,r)$, but
that is not true in $X$ if $y$ is not a fixed point of $f\circ g$.
Thus $i_G:X\to G(X,r)$ is not injective in general. 
Note that, in this case, $r$ is involutive if and only if $g=f^{-1}$.    
\end{example}

In other words, due to its defining relations, some of the generators of $G(X,r)$ (resp. 
of $A(X,r)$) might become identified. We must be very careful with this fact during our 
proofs. 

\paragraph{} The next lemma shows that, in some sense, it is natural to consider the last two groups, 
because then we can extend the maps $f$ and $g$ to define actions of them. 

\begin{lemma}[{\cite[Theorem~2.4]{Soloviev}}]\label{actionsfg}
The map $f: X\to \sym_X$, $x\mapsto f_x$, induces a unique morphism of groups
$f: G(X,r)\to\sym_X$ such that $f_{i_G(x)}(y)=f_x(y)$ for any $x,y\in X$. It also induces a unique morphism
$\lambda: G(X,r)\to \aut(A(X,r))$ such that $\lambda_{i_G(x)}(i_A(y))=i_A(f_x(y))$
for any $x,y\in X$. Moreover, the map $g:X\to\sym_X$, $x\mapsto g_x$, induces 
 a unique anti-morphism of groups $g: G(X,r)\to\sym_X$ such that $g_{i_G(x)}(y)=g_x(y)$ for any 
 $x,y\in X$.
\end{lemma}

Using these actions, we can prove the following fundamental theorem, which provides a structure of 
skew left brace over $G(X,r)$. Here, the group $A(X,r)$ plays the role of the 
star group, and the map $\lambda$ that we have just defined is the action in Lemma~\ref{propieLambda}.
 This theorem is really important for us, because the other results of this 
paper will be a direct consequence 
of this fundamental theorem.

\begin{theorem}[{\cite[Theorem~2.5]{Soloviev}, \cite[                                                                                                                                                                                                                                   Theorem~9]{LYZ}}]\label{braceStruct}
Let $(X,r)$ be any non-degenerate solution. Then, we can define a product $\star$ 
over $G(X,r)$ such that $(G(X,r),\cdot,\star)$ is a skew left brace, and such that
there exists a group isomorphism $\varphi: (G(X,r),\star)\to A(X,r)$ with $\varphi(i_G(x))=i_A(x)$ for any $x\in X$. 
This operation is given by $b\star b'=b\cdot \varphi^{-1}(\lambda^{-1}_b(\varphi(b')))$, 
for any $b,b'\in G(X,r)$, where $\lambda$ is defined in the last lemma. 
\end{theorem}

In particular, $i_G(x)=i_G(y)$ for some $x,y\in X$ if and only if $i_A(x)=i_A(y)$. Thus, from now on, 
we will identify $(G(X,r),\star)$ with $A(X,r),$ and we will use the notation $i=i_G=i_A$. 

The next technical lemma, that we will need later on, is an example of 
the kind of computations that one encounters working with this new operation 
over $G(X,r)$. 

\begin{lemma}\label{lemmaxb}
Let $(X,r)$ be a non-degenerate solution. Then, for any $x\in X$ and for any $b\in G(X,r)$, 
$i(x)\cdot b=\lambda_x(b)\cdot i(g_b(x))$. 
\end{lemma}
\begin{proof}
Recall the action $\gamma$ appearing in Lemma~\ref{propieg}. In any skew left brace, it is true that 
$a\cdot b=\lambda_a(b)\cdot \gamma_b(a)$ for any pair of elements $a,b$. In the case of $G(X,r)$, it is easy to check (using 
the relations of $A(X,r)$ and Lemmas~\ref{actionsfg} and \ref{braceStruct}) that $\gamma_b(i(x))=i(g_b(x))$ for any 
$x\in X$ and $b\in G(X,r)$. Hence $i(x)\cdot b=\lambda_x(b)\cdot \gamma_b(x)=\lambda_x(b)\cdot i(g_b(x)).$
\end{proof}

Next we define another fundamental group for the study of non-degenerate solutions. It can be defined 
directly using the maps $f_x$ and $g_x$, but it is also a quotient of the structure group $G(X,r)$. 
But first, we need a new action of $G(X,r)$ over $X$, which is defined using $\lambda$ and $g$. 
Abusing of the notation, we denote $i(x)$, $x\in X$, by $x$ when it is clear that we are working inside 
$G(X,r)$ and not in $X$. 

\begin{lemma}
The map $\stackrel{\sim}{g}: G(X,r)\to\sym_X$ defined by $\stackrel{\sim}{g}_b(x):=g_{\lambda^{-1}_{x}(b)}(x)$, 
for any $b\in G(X,r)$ and $x\in X$, is an anti-morphism of groups. 
\end{lemma}
\begin{proof}
First we shall check that, for any $b,b'\in G(X,r)$, the maps $\stackrel{\sim}{g}_b,\stackrel{\sim}{g}_{b'}: X\to X$ 
satisfy $\stackrel{\sim}{g}_b\circ \stackrel{\sim}{g}_{b'}=\stackrel{\sim}{g}_{b'\cdot b}$ . 
Indeed, for any $x\in X$, the left-hand side is equal to 
\begin{align*}
\stackrel{\sim}{g}_b( \stackrel{\sim}{g}_{b'}(x))&=\stackrel{\sim}{g}_b(g_{\lambda^{-1}_{x}(b')}(x))
=g_{\lambda^{-1}_{g_{\lambda^{-1}_x(b')}(x)}(b)} g_{\lambda^{-1}_x(b')}(x)\\
&=g_{\lambda^{-1}_x(b')\cdot \lambda^{-1}_{g_{\lambda^{-1}_x(b')}(x)}(b)}(x), 
\end{align*}
and the right-hand side is equal to 
\begin{align*}
\stackrel{\sim}{g}_{b'\cdot b}(x)&=g_{\lambda^{-1}_{x}(b'\cdot b)}(x)=g_{\lambda^{-1}_{x}(b'\star \lambda_{b'}(b))}(x)\\
&=g_{\lambda^{-1}_{x}(b')\star \lambda^{-1}_{x}(\lambda_{b'}(b))}(x)=
g_{\lambda^{-1}_{x}(b')\cdot \lambda^{-1}_{\lambda^{-1}_{x}(b')} \lambda^{-1}_{x}\lambda_{b'}(b)}(x)\\
&=g_{\lambda^{-1}_x(b')\cdot \lambda^{-1}_{g_{\lambda^{-1}_x(b')}(x)}(b)}(x),
\end{align*}
where in the last step we apply that the equality $x\cdot \lambda^{-1}_x(b')=b'\cdot g_{\lambda^{-1}_x(b')}(x)$ is satisfied 
in $G(X,r)$ (Lemma~\ref{lemmaxb}). So we conclude $\stackrel{\sim}{g}_b( \stackrel{\sim}{g}_{b'}(x))=\stackrel{\sim}{g}_{b'\cdot b}(x).$

Finally, for any $b\in G(X,r)$, since $\stackrel{\sim}{g}_b\circ \stackrel{\sim}{g}_{b^{-1}}=
\stackrel{\sim}{g}_{b^{-1}\cdot b}=
\stackrel{\sim}{g}_{1}=\id=
\stackrel{\sim}{g}_{b}\circ \stackrel{\sim}{g}_{b^{-1}}$, we conclude that 
$\stackrel{\sim}{g}_{b}\in \sym_X$. 
\end{proof}

\begin{definition}
Let $(X,r)$ be a non-degenerate solution. The permutation group of 
$(X,r)$, denoted by $\mathcal{G}(X,r)$, is defined by
\begin{align*}
\mathcal{G}(X,r)&:=\left\langle \left(f_x,~\stackrel{\sim}{g}^{-1}_x\right):x\in X\right\rangle\\
&=\left\{\left(f_a,~\stackrel{\sim}{g}^{-1}_a\right):a\in G(X,r)\right\}
\leq \sym_X\times\sym_X.
\end{align*}
\end{definition}

In fact, the permutation group of any non-degenerate solution is not only a group, but also has the structure of a skew 
left brace.

\begin{theorem}[{\cite[Theorem~2.6]{Soloviev}}]\label{bracePerm}
Let $(X,r)$ be a non-degenerate solution. Then, the
permutation group $\mathcal{G}(X,r)$ has a structure of skew left brace.
\end{theorem}
\begin{proof}
Let $I=\{b\in G(X,r): f_b=\stackrel{\sim}{g}_b=\id\}$. In other words, $I$ is the kernel of the morphism 
$\varphi: G(X,r)\to\sym_X\times\sym_X$, $b\mapsto (f_b,\stackrel{\sim}{g}^{-1}_b)$. We are going to show that 
$I$ is a brace ideal of $G(X,r)$. We already know that it is a normal subgroup of $G(X,r)$. 
In fact, $I$ is contained in $\soc(G(X,r))=\{b\in G(X,r): \lambda_b=\id,~a\star b=b\star a\text{ for any }a\in G(X,r)\}$. 
If $b\in I$, then $f_b=\id$ and $\stackrel{\sim}{g}_b=\id$. The first property implies that $\lambda_b=\id$. 
On the other hand, by Lemma~\ref{lemmaxb}, $x\cdot \lambda^{-1}_x(b)=b\cdot g_{\lambda^{-1}_x(b)}(x)=
b~\cdot \stackrel{\sim}{g}_b(x)=b\cdot x$. 
Hence, $x\star b=x\cdot \lambda^{-1}_x(b)=b\cdot x=b\star \lambda_b(x)=b\star x$. Since 
$b$ commutes with all the generators of $G(X,r)$, we conclude that $a\star b=b\star a$ for any $a\in G(X,r)$, thus $b\in\soc(G(X,r))$. 
Then $I$ is closed by lambda maps because, by Lemma~\ref{SocleKernel}(i), if $b\in I\subseteq \soc(G(X,r))$, 
$\lambda_{a}(b)=a\cdot b\cdot a^{-1}\in I$ for any $a\in G(X,r)$. 
It is a normal subgroup of $(G(X,r),\star)$ because $I$ is contained in the socle, which is 
contained in the center of $(G(X,r),\star)$. 

By definition of $\mathcal{G}(X,r)$, we have that $\mathcal{G}(X,r)=\im(\varphi)\cong (G(X,r)/I,\cdot).$
Using this isomorphism, the star product of $G(X,r)/I$ can also be defined over 
$\mathcal{G}(X,r)$. 
\end{proof}

Observe that the star operation in $\mathcal{G}(X,r)$ is given by 
$$\left(f_a,~\stackrel{\sim}{g}_a^{-1}\right)\star \left(f_b,~\stackrel{\sim}{g}_b^{-1}\right)
=\left(f_a\circ f_{f^{-1}_a(b)},~\stackrel{\sim}{g}^{-1}_a\circ \stackrel{\sim}{g}^{-1}_{f^{-1}_a(b)}\right),$$
for $a,b\in G(X,r)$. We will show in Section~\ref{sectConstSol} that $\mathcal{G}(X,r)$ with its brace structure 
is the key piece of information to study non-degenerate solutions, because we are going to show that it is possible 
to recover not only $(X,r)$ but also all the solutions with the same permutation group. Thus from a given 
permutation group of a solution we are able to construct an infinite number of solutions.

In this paper, we have defined the permutation group of a non-degenerate solution following the original definition
of \cite{Soloviev} and \cite{LYZ}, used also in \cite{Takeuchi}. In \cite{GV}, they define $\mathcal{G}(X,r)$ 
as $G(X,r)/\soc(G(X,r))$. This does not coincide with the definition of \cite{LYZ,Soloviev}, because
$I$ is always contained in $\soc(G(X,r))$, but in some cases, $I\neq \soc(G(X,r))$, as the following example shows.  

\begin{example}
Let $X=\Z$ be the set of integer numbers, and let $f: X\to X$ be the map $f(x)=x+1$, $x\in X$. 
Then, the map $r:X^2\to X^2$, where $r(x,y)=(f(y),f(x))$, is a non-degenerate solution of the 
Yang-Baxter equation. Its structure group is isomorphic to the group $G$ given by the presentation
$$
G=\left\langle x_i,i\in\Z : x_i\cdot x_j=x_{j+1}\cdot x_{i+1}\right\rangle, 
$$
but, due to the relations of $G$, $x_{i-1}=x_{i+1}$ for any $i\in \Z$, so 
$G$ is isomorphic to the group 
$$
\left\langle x_0,x_1: x_0^2=x_1^2\right\rangle. 
$$
In the same way, one shows that the derived group $A(X,r)$ is isomorphic to $\Z^2$.

In this case, $f_x=f$ and $\stackrel{\sim}{g}^{-1}_x=f^{-1}$ for any $x\in X$. 
So $$\mathcal{G}(X,r)=\left\langle\left(f_x,~\stackrel{\sim}{g}^{-1}_x\right):x\in X\right\rangle
=\left\langle\left(f,~f^{-1}\right)\right\rangle\cong \Z.$$ 
But, on the other hand, $\lambda:A(X,r)\to A(X,r)$ induced by $f$ is given by $\lambda(x_0)=x_1$ and $\lambda(x_1)= x_0$, 
and the morphism $h: (G(X,r),\cdot)\to \aut(G(X,r),\star)$ of Lemma~\ref{SocleKernel} is given by $h_a(b)=a\star \lambda_a(b)\star a^\star=
\lambda_a(b)=\lambda(b)$. 
Therefore, by Lemma~\ref{SocleKernel},
$G(X,r)/\soc(G(X,r))\cong\Z/(2)$, which finally shows $$\mathcal{G}(X,r)=G(X,r)/I\not\cong G(X,r)/\soc(G(X,r))$$
in this example; equivalently, $I\neq \soc(G(X,r))$. 
\end{example}

In fact, the skew left brace $G(X,r)/\soc(G(X,r))$ of \cite{GV} is a quotient of our $\mathcal{G}(X,r)$. 
The definition of \cite{GV} has the disadvantage that the results of Section~\ref{sectConstSol} about construction of solutions 
are not possible with their definition; this is the reason why we prefer to use the more classical definition of \cite{LYZ,Soloviev}. 
A natural question in view of the last examples is the following.

\begin{openProblem}
Is it true that $I=\soc(G(X,r))$ when $X$ is finite?
\end{openProblem}

\subsection{Some results about $G(X,r)$ and $A(X,r)$}\label{sectGroupTheory}
Since $I$ is contained in $\soc(G(X,r))$, 
which in turn is contained in the center of $A(X,r)$, it is an abelian subgroup of $G(X,r)$, 
and since $\mathcal{G}(X,r)\cong G(X,r)/I$ 
is a subgroup of $\sym_X\times\sym_X$, we obtain immediately the following result about 
the group structure of $G(X,r)$ when $X$ is finite. 

\begin{corollary}[{\cite[Theorem~2.6]{Soloviev}, \cite[Proposition~10]{LYZ}}]\label{abelianbyfinite}
Let $(X,r)$ be a finite non-degenerate solution. Then, $G(X,r)$ and $A(X,r)$ are finitely generated 
abelian-by-finite groups. Moreover, $A(X,r)$ has a central subgroup of finite index. 
\end{corollary}

In fact, in \cite[Section~2.4]{Soloviev}, it is described the rank of the finitely generated 
abelian group $\soc(G(X,r))$. It is always less than or equal to $n=\ord{X}$, and equality holds
if and only if $(X,r)$ is involutive. 

We can also determine when $A(X,r)$ is abelian. If $(X,r)$ is a non-degenerate solution, we define 
$\overline{r}: i(X)^2\to i(X)^2$ by $\overline{r}(i(x),i(y))=(i(f_x(y)),~i(g_y(x)))$, for any 
$x,y\in X$. It is not difficult to check that it is a well-defined map, and that it defines a 
non-degenerate solution over $i(X)$. 

\begin{proposition}[{\cite[Proposition~4]{LYZ}}]
Let $(X,r)$ be a non-degenerate solution. The following conditions are equivalent:
\begin{enumerate}[(i)]
\item $A(X,r)$ is abelian.
\item $A(X,r)$ is free abelian with basis $i(X)$.
\item $\overline{r}$ is involutive. 
\end{enumerate}
\end{proposition}
\begin{proof}
The implications $(i)\Leftrightarrow (iii)$ and $(ii)\Rightarrow (i)$ are straightforward. 
For the implication $(i)\Rightarrow (ii)$, the map 
$$
\begin{array}{ccl}
X& \longrightarrow& \langle i(X):i(x)\cdot i(y)=i(y)\cdot i(x)\rangle=H\cong \Z^{(i(X))}\\
x& \mapsto& i(x)
\end{array}
$$
induces a morphism of groups $A(X,r)\to H$ because it preserves the relations of $A(X,r)$: 
since $A(X,r)$ is abelian, we know that $\overline{r}$ is involutive, which implies 
$i(f_{f_x(y)}g_y(x))=i(x)$ for any $x,y\in X$. Thus, if we apply $i$ to a relation
$f_x(y)^\star\star x^\star\star f_x(y)\star f_{f_x(y)}g_y(x)$ of $A(X,r)$, we obtain
$i(f_x(y))^\star\star i(x)^\star\star i(f_x(y))\star i(f_{f_x(y)}g_y(x))=
i(f_x(y))^\star\star i(x)^\star\star i(f_x(y))\star i(x)
=i(f_x(y))^\star\star i(x)^\star\star i(x)\star i(f_x(y)).$
This morphism is in fact bijective, with inverse morphism induced by the inclusion map 
$i(X)\to A(X,r)$. 
\end{proof}

In the finite case, it is also possible to characterize when $A(X,r)$ is torsion-free.

\begin{proposition}
Let $(X,r)$ be a finite non-degenerate solution. Then, 
$A(X,r)$ is torsion-free if and only if $A(X,r)$ is abelian. 
\end{proposition}
\begin{proof}
By the last proposition, it is enough to show that $A(X,r)$ is torsion-free if and only if $A(X,r)$ is free abelian. 
The only part that we have to prove is the implication from left to right. Using Corollary~\ref{abelianbyfinite}, if $(X,r)$ is finite, 
$A(X,r)$ is a finitely generated group that contains an abelian subgroup of finite index. Recall that finite index subgroups 
of finitely generated groups are also finitely generated, see \cite[Proposition~II.4.2]{LyndonSchupp}. Then, if $A(X,r)$ is torsion-free, 
we conclude that it contains a central finitely generated free abelian group of finite index. 
By \cite[Proposition~III.7.3]{LyndonSchupp}, 
we conclude that $A(X,r)$ is free abelian. 
\end{proof}

The following is a natural question after this result. We only know the answer for $(X,r)$ involutive \cite[Corollary~8.2.7]{JObook}. 

\begin{openProblem}
Determine when $G(X,r)$ is torsion-free, for $(X,r)$ a finite non-degenerate solution. 
\end{openProblem}

\subsection{Generating all the non-degenerate solutions associated to a skew left brace}\label{sectConstSol}

If two non-degenerate solutions $(X,r)$ and $(Y,s)$ are isomorphic, then we have $G(X,r)\cong G(Y,s)$ as braces, so 
$\mathcal{G}(X,r)\cong \mathcal{G}(Y,s)$. But the converse is not true: non-isomorphic solutions may have isomorphic 
permutation groups (as braces). In this section, we are going to show that, given a skew left brace $B$, 
it is possible to reconstruct all the non-degenerate solutions $(X,r)$ such that $\mathcal{G}(X,r)\cong B$ only using 
the brace structure of $B$. 

As a corollary, we obtain that for any skew left brace $B$, there exists a non-degenerate solution $(X,r)$ such that 
$B\cong\mathcal{G}(X,r)$. In other words, skew left braces and permutation groups are equivalent concepts. 
 We think that this and the result about construction of solutions from braces justify that 
skew left braces are the right algebraic structure to study non-degenerate solutions. 

Consider the semidirect product of groups $(G(X,r),\star)\rtimes(G(X,r),\cdot)$ with respect to the group action
$\lambda: (G(X,r),\cdot)\to\aut(G(X,r),\star)$. 
One can show that the map 
$$
\begin{array}{cccc}
\overline{\sigma}: & (G(X,r),\star)\rtimes (G(X,r),\cdot)& \to & \sym_X\\
 & (a,b)& \mapsto & \stackrel{\sim}{g}^{-1}_a\circ f^{-1}_a\circ f_b
\end{array}
$$ 
is a morphism of groups. 
Observe that now the morphism $(G(X,r),\cdot)\to \sym_X\times\sym_X$ given by $a\mapsto \left(f_a,\stackrel{\sim}{g}^{-1}_a\right)$  
(appearing in the proof of Theorem~\ref{bracePerm})
can be rewritten as $a\mapsto \left(f_a,\stackrel{\sim}{g}^{-1}_a\right)=(\overline{\sigma}_{(1,a)},\overline{\sigma}_{(a,a)}).$ 
Moreover, to relate it to the material of Section~\ref{sectSkewBraces}, note that $\Theta_{(a,b)}(i(x))=i(\overline{\sigma}_{(a,b)}(x))$, 
$a,b\in B$, $x\in X$, where $\Theta$ is defined in Lemma~\ref{propieAction}.

Let $I=\{a\in G(X,r): f_a=\stackrel{\sim}{g}_a=\id_X\}$. Recall that it is an ideal of $G(X,r)$ and 
$G(X,r)/I\cong \mathcal{G}(X,r)$ as braces. Note that 
$$N:=\{(a,b)\in (G(X,r),\star)\rtimes (G(X,r),\cdot): a,b\in I\}$$
is a normal subgroup of $(G(X,r),\star)\rtimes (G(X,r),\cdot)$ contained in $\ker(\overline{\sigma})$, and that 
$$[(G(X,r),\star)\rtimes (G(X,r),\cdot)]/N\cong (\mathcal{G}(X,r),\star)\rtimes (\mathcal{G}(X,r),\cdot).$$ 
Hence, $\overline{\sigma}$ induces a morphism 
\begin{align}\label{actionSemid}
\begin{array}{cccc}
\sigma: & (\mathcal{G}(X,r),\star)\rtimes (\mathcal{G}(X,r),\cdot)& \to & \sym_X.
\end{array}
\end{align}

This is a fundamental action of $\mathcal{G}(X,r)$. Inspired by this, the next proposition gives a first way to 
construct non-degenerate solutions from a given skew left brace. The result is known \cite[Theorem~2.7]{Soloviev}, 
but we give a restatement in terms of brace theory. 
The analogous result in the setting of non-degenerate involutive solutions appeared in 
\cite[pages 182--183]{ESS}. Note also that \cite[Theorem~4.6]{BDG2} 
(see also \cite[Corollary D]{NirBenDavid}) gives a cohomological version of it in the involutive case.  

In all this section, the semidirect product of groups $(B,\star)\rtimes (B,\cdot)$, where $B$ is a skew 
left brace, is defined with respect to the group action $\lambda: (B,\cdot)\to\aut(B,\star)$. 
Hence, the product in $(B,\star)\rtimes (B,\cdot)$ is given by $(a,b)(a',b')=(a\star\lambda_b(a'),~b\cdot b')$. 

\begin{proposition}\label{BenDavidgen}
Let $B$ be a skew left brace, and let $X$ be a set. Suppose that $\eta: X\to B$ is a map such that 
$\left\langle \eta(X)\right\rangle_\star=B$, and that 
$\sigma: (B,\star)\rtimes (B,\cdot)\to\sym_X$ is a morphism 
of groups such that $(B,\cdot)\to \sym_X\times\sym_X,$ $b\mapsto (\sigma_{(1,b)},\sigma_{(b,b)}),$ is injective, 
and 
\begin{align}\label{hypConstSol}
\eta(\sigma_{(a,b)}(x))=a\star \lambda_b(\eta(x))\star a^\star,
\end{align} for any $a,b\in B$ and $x\in X$. 
Define a map
$$
\begin{array}{cccc}
r:& X\times X& \to& X\times X\\
& (x,y)& \mapsto & (f_x(y),g_y(x)),
\end{array}
$$ 
where 
\begin{align*}
f_x(y)&:=\sigma_{(1,\eta(x))}(y)\text{, and }\\
g_y(x)&:=\sigma^{-1}_{(\eta(\sigma_{(1,\eta(x))}(y)),~\eta(\sigma_{(1,\eta(x))}(y)))}(x)=
\sigma^{-1}_{(\lambda_{\eta(x)}(\eta(y)),~\lambda_{\eta(x)}(\eta(y)))}(x).
\end{align*} 
Then $(X,r)$ is a non-degenerate 
solution such that $\mathcal{G}(X,r)\cong B$ as braces. 

Moreover, any non-degenerate solution $(Z,t)$ such that $\mathcal{G}(Z,t)\cong B$ is of this form. 
\end{proposition}

As we have seen in the previous proposition, given a skew left brace
$B$, to construct a non-degenerate solution $(X,r)$ such that
$\mathcal{G}(X,r)\cong B$, it is enough to find two maps $\eta\colon
X\longrightarrow B$ and $\sigma\colon (B,\star)\rtimes (B,\cdot)\longrightarrow
\sym_X$ satisfying some properties. Next, we show
how to construct the sets $X$ and the maps $\eta$ and $\sigma$ using
only information about the structure of the skew left brace $B$.

Let $G=(B,\star)\rtimes (B,\cdot)$. Consider the action $\Theta : G \rightarrow \aut(B,\star)$ 
defined in Lemma~\ref{propieAction}. Then, $B$ and $X$ become $G$-sets via $\Theta$ and $\sigma$, respectively.
Observe that $$\eta(\sigma_{(a,b)}(x))=a\star \lambda_b(\eta(x))\star a^\star=\Theta_{(a,b)}(\eta(x)),$$ so $\eta$ becomes 
a $G$-set morphism.

We use the structure of $G$-sets as in \cite[Chapter~1.7]{Suzuki}. We only need the following well-known properties:
\begin{enumerate}[(a)]
\item For any subgroup $H$ of $G$, the set $G/H$ of left cosets of $H$ in $G$ is a $G$-set with action $g(xH)=(gx)H$, $g,x\in G$. 
\item Any $G$-set decomposes as a disjoint union of orbits. Any orbit is isomorphic, as a $G$-set, to $G/H$ for some $H\leq G$; 
in fact, $H$ can be taken to be the stabilizer of one element of the orbit.
\item The kernel of the morphism $G\rightarrow\sym_{G/H}$ is equal to the core of $H$, defined by $\core_G(H)=\bigcap_{g\in G} gHg^{-1}$. 
It is the maximal normal subgroup of $G$ contained in $H$. In  general, the kernel of $G\rightarrow\sym_{\bigsqcup_i G/H_i}$ is equal to 
$\bigcap_i \core_G(H_i)=\core_G(\bigcap_i H_i).$ 
\item If $\eta: X\to Y$ is a surjective morphism of $G$-sets, and $Y$ decomposes as the union of orbits $Y=\bigsqcup_{i\in I} Y_i$, then 
$X$ decomposes as the union of orbits $X=\bigsqcup_{i\in I}\bigsqcup_{j\in J_i} X_{i,j}$, where $\eta(X_{i,j})=Y_i$ for any $i\in I$, $j\in J_i$. 
\end{enumerate}

Using these facts, we can translate the above theorem into something that depends only on 
the algebraic structure of $B$. We describe now the necessary ingredients. 
The stabiliser of $a\in B$ by the action $\Theta$ is
denoted $\St(a)$. For $a\in B$, let $B_a=\{\Theta_{(b,c)}(a)\mid b,c\in
B\}$ be the orbit of $a$, and let $\mathcal{O}=\{ B_a\mid a\in B\}$
be the set of orbits of the action $\Theta$.  For each $i\in
\mathcal{O}$, choose an element $a_i\in i$. Let $I$ be a subset of
$\mathcal{O}$, such that $Y=\bigcup_{i\in I}i$ satisfies $B=\langle
Y\rangle_\star$, the subgroup of $(B,\star)$ generated by $Y$. For each $i\in
I$, let $J_i$ be a non-empty set and  let $\{ K_{i,j} \}_{j\in J_i}$ be
a family of subgroups of $\St(a_i)$ such that
$$\left\{a\in B:(1,a),(a,a)\in\bigcap_{i\in I}\bigcap_{j\in J_i}\core_G(K_{i,j})\right\}=\{ 1\}.$$
Note that if one of the subgroups $K_{i,j}$ is trivial, then this
last condition is satisfied.

\begin{theorem}\label{constSol}
 With the above notation, define the set $$X:=\bigsqcup_{i\in I}\bigsqcup_{j\in
J_i}G/K_{i,j}$$ as the disjoint union of the sets of left cosets
$G/K_{i,j}$. Then, $(X,r)$ is a non-degenerate solution such that $\mathcal{G}(X,r)\cong B$ as
braces, with $r$ being the map
$$
\begin{array}{cccc}
r\colon& X\times X &\longrightarrow & X\times X\\
& (x,~y) &\mapsto
&(f_{x}(y),~g_{y}(x)),
\end{array}
$$
 where, for any $x=(b_1,c_1)K_{i_1,j_1}$ and $y=(b_2,c_2)K_{i_2,j_2}$ in $X$,
\begin{align*}
f_{x}(y)&:=(1,\Theta_{(b_1,c_1)}(a_{i_1}))\cdot(b_2,c_2)K_{i_2,j_2}\\
&=(1,b_1\star \lambda_{c_1}(a_{i_1})\star b_1^\star)\cdot(b_2,c_2)K_{i_2,j_2}
\end{align*}
 and 
\begin{align*}
g_{y}(x)&:=\left(\Theta_{(1,\Theta_{(b_1,c_1)}(a_{i_1}))}\Theta_{(b_2,c_2)}(a_{i_2}),~
\Theta_{(1,\Theta_{(b_1,c_1)}(a_{i_1}))}\Theta_{(b_2,c_2)}(a_{i_2})\right)^{-1}(b_1,c_1)K_{i_1,j_1}.
\end{align*}

Moreover, any solution $(Z,t)$ with $\mathcal{G}(Z,t)\cong B$ as
 braces is isomorphic to  such a solution.
\end{theorem}
\begin{proof}
We are only going to sketch the main steps of the proof. For the first part of the theorem, we define 
$\eta: X\to B$ as $\eta((b,c)K_{i,j}):=\Theta_{(b,c)}(a_i)$, and we also define $\sigma: G\to \sym_X$ 
to be the natural action of $G$ on $X$ given by left multiplications on the cosets in $G/K_{i,j}$; 
i.e. $\sigma_{(b,c)}((u,v)K_{i,j}):=(b,c)(v,u)K_{i,j}=(b\star \lambda_c(u),~c\cdot v)K_{i,j}.$ Then, we only have
to check that the hypothesis of Proposition~\ref{BenDavidgen} are fulfilled for these $\eta$ and $\sigma$, and that the maps 
$f$ and $g$ defined in the statement of this theorem coincide with the ones defined in Proposition~\ref{BenDavidgen}.

To prove the second part, we only have to use the fact that any solution is of the form given by Proposition~\ref{BenDavidgen}, together 
with the facts about $G$-sets that we have mentioned previously. 

\end{proof}

As a consequence of Theorem~\ref{constSol}, given a skew left brace $B$, to
construct all the non-degenerate solutions $(X,r)$ of the Yang-Baxter equation such that
$\mathcal{G}(X,r)\cong B$ as skew left braces one can proceed as follows:
\begin{enumerate}
\item Find the decomposition of $B$ as disjoint union of orbits, 
$B=\bigcup_{i\in K} B_i$, by the action $\Theta\colon (B,\star)\rtimes (B,\cdot)\longrightarrow\aut(B,\star)$.
Then choose one element $a_i$ in each orbit $B_i$ for all $i\in K$.

\item Find all the subsets $I$ of $K$ such that the subset
$Y=\bigcup_{i\in I}B_i$ generates the star group of $B$.

\item Given such an $Y$, find for each $i\in I$ a non-empty family $\{K_{i,j}\}_{j\in
J_i}$ of subgroups of $\St(a_i)$ such that 
$$\left\{a\in B:(1,a),(a,a)\in\bigcap_{i\in I}\bigcap_{j\in J_i}\core_G(K_{i,j})\right\}=\{ 1\}.$$ 
Note that the $K_{i,j}$
could be equal for different $(i,j)$.

\item Construct a solution as in the statement of
Theorem~\ref{constSol} using the families $\{K_{i,j}\}_{j\in J_i}$, for
$i\in I$.
\end{enumerate}

Note that by Theorem~\ref{constSol}, any non-degenerate solution $(X,r)$
such that $\mathcal{G}(X,r)\cong B$  (as braces) is isomorphic
to one constructed in this way. It could happen that different
solutions constructed in this way from a skew left brace $B$
are in fact isomorphic. In the next result we characterize when two
of these solutions are isomorphic.

Let $B$ be a skew left brace and let $\mathcal{O}$, $I$, $a_i$, $J_i$,
$K_{i,j}$ be as before. Let $(X,r)$ be
the solution of the statement of Theorem~\ref{constSol}.

Let $I'\subseteq \mathcal{O}$ such that $Y'=\bigcup_{i'\in I'}i'$
satisfy $B=\langle Y'\rangle_\star$.
For each $i'\in I'$, let $\{ L_{i',j'} \}_{j'\in J'_{i'}}$ be a
non-empty family of subgroups of $\St(a_{i'})$ such that
$$\left\{a\in B:(1,a),(a,a)\in\bigcap_{i'\in I'}\bigcap_{j'\in J'_{i'}}\core_G(L_{i',j'})\right\}=\{ 1\}.$$
Let $(X',r')$ be the corresponding solution defined as in
the statement of Theorem~\ref{constSol}.
We shall characterize when $(X,r)$ and $(X',r')$ are isomorphic in
the following result.

\begin{theorem}\label{isomSol}
The solutions $(X,r)$ and $(X',r')$ are isomorphic if and only if
there exist an automorphism $\psi$ of the skew left brace $B$, a
bijective map $\alpha\colon I\rightarrow I'$, a bijective map
$\beta_i\colon J_i\rightarrow J'_{\alpha(i)}$ and $y_{i,j},z_{i,j}\in B$,
for each $i\in I$ and $j\in J_i$, such that
$$\psi(a_i)=\Theta_{(y_{i,j},z_{i,j})}(a_{\alpha(i)})
\quad\mbox{and}\quad (\psi\times\psi)(K_{i,j})=(y_{i,j},z_{i,j})L_{\alpha(i),\beta_i(j)}(y_{i,j},z_{i,j})^{-1},$$
for all $i\in I_1$ and $j\in J_i$.
\end{theorem}
\begin{proof}
For the if part, 
we define a map $F:X\to X'$ by $$F((b,c)K_{i,j}):=(\psi(b),\psi(c))(y_{i,j},z_{i,j})L_{\alpha(i),\beta_i(j)},$$ for any $b,c\in B$. 
It is not difficult to check that $F$ is well-defined, and that it is an isomorphism between $(X,r)$ and $(X',r')$. 

Conversely, suppose that there exists an isomorphism $F\colon
X\rightarrow X'$ of the solutions $(X,r)$ and $(X',r')$. We can
write $F((b,c) K_{i,j})=\varphi((b,c) K_{i,j})L_{\alpha(b,c,i,j),\beta(b,c,i,j)}$,
for some maps $\varphi\colon X\rightarrow G$, $\alpha\colon
X\rightarrow I'$ and $\beta \colon X\rightarrow \bigcup_{i'\in
I'}J'_{i'}$. We shall prove that $\alpha(b,c,i,j)=\alpha(1,1,i,k)$ and
$\beta(b,c,i,j)=\beta(1,1,i,j)$, for all $b,c\in B$, $i\in I$ and $j,k\in
J_i$. Since $F$ is a morphism of solutions of the Yang-Baxter equation, the condition $F(f_x(y))=f'_{F(x)}(F(y))$ 
implies that
\begin{align}\label{F}
\lefteqn{F((1,\Theta_{(b_1,c_1)}(a_{i_1}))(b_2,c_2)K_{i_2,j_2})}\nonumber\\
&=(1,\Theta_{\varphi((b_1,c_1)K_{i_1,j_1})}(a_{\alpha(b_1,c_1,i_1,j_1)}))F((b_2,c_2)K_{i_2,j_2}),
\end{align}
for all $b_1,c_1,b_2,c_2\in B$, $i_1,i_2\in I$,  $j_1\in J_{i_1}$ and
$j_2\in J_{i_2}$. Hence
\begin{align*}\lefteqn{\varphi((\lambda_{\Theta_{(b_1,c_1)}(a_{i_1})}(b_2),\Theta_{(b_1,c_1)}(a_{i_1})c_2)
K_{i_2,j_2})}\\
&\cdot L_{\alpha(\lambda_{\Theta_{(b_1,c_1)}(a_{i_1})}(b_2),\Theta_{(b_1,c_1)}(a_{i_1})c_2,i_2,j_2),~
\beta(\lambda_{\Theta_{(b_1,c_1)}(a_{i_1})}(b_2),\Theta_{(b_1,c_1)}(a_{i_1})c_2,i_2,j_2)}
\\
&=(1,\Theta_{\varphi((b_1,c_1)K_{i_1,j_1})}(a_{\alpha(b_1,c_1,i_1,j_1)}))\varphi((b_2,c_2)K_{i_2,j_2})L_{\alpha(b_2,c_2,i_2,j_2),\beta(b_2,c_2,i_2,j_2)},
\end{align*}
for all $b_1,c_1,b_2,c_2\in B$, $i_1,i_2\in I$,  $j_1\in J_{i_1}$ and
$j_2\in J_{i_2}$. Thus
$$\alpha(\lambda_{\Theta_{(b_1,c_1)}(a_{i_1})}(b_2),\Theta_{(b_1,c_1)}(a_{i_1})\cdot c_2,i_2,j_2)=\alpha(b_2,c_2,i_2,j_2)$$
and
$$\beta(\lambda_{\Theta_{(b_1,c_1)}(a_{i_1})}(b_2),\Theta_{(b_1,c_1)}(a_{i_1})\cdot c_2,i_2,j_2)=\beta(b_2,c_2,i_2,j_2).$$ Since
$B=\langle Y\rangle_\star$ and $Y$ is $G$-invariant (by the action
$\Theta$), we know that $Y$ also generates the multiplicative group
of $B$. Therefore 
\begin{align}\label{alphac2}
\alpha(b_2,c_2,i_2,j_2)=\alpha(\lambda_a(b_2),a\cdot c_2,i_2,j_2)
\end{align} 
and
\begin{align*}
\beta(b_2,c_2,i_2,j_2)=\beta(\lambda_a(b_2),a\cdot c_2,i_2,j_2),
\end{align*} 
for any $a\in B$. 

Using in a similar way the condition $F(g_x(y))=g'_{F(x)}(F(y))$, we get that
\begin{align}\label{F2}
\lefteqn{F((\Theta_{(b_1,c_1)}(a_{i_1}),\Theta_{(b_1,c_1)}(a_{i_1}))^{-1}(b_2,c_2)K_{i_2,j_2})}\nonumber\\
&=(\Theta_{\varphi((b_1,c_1)K_{i_1,j_1})}(a_{\alpha(b_1,c_1,i_1,j_1)}),\Theta_{\varphi((b_1,c_1)K_{i_1,j_1})}(a_{\alpha(b_1,c_1,i_1,j_1)}))^{-1}F((b_2,c_2)K_{i_2,j_2}),
\end{align}
for all $b_1,c_1,b_2,c_2\in B$, $i_1,i_2\in I$,  $j_1\in J_{i_1}$ and
$j_2\in J_{i_2}$,
which in turn implies that 
\begin{align}\label{alphab2}
\alpha(b_2,c_2,i_2,j_2)=\alpha(a^{-1}\cdot b_2,a^{-1}\cdot c_2,i_2,j_2)
\end{align} and
\begin{align*}
\beta(b_2,c_2,i_2,j_2)=\beta(a^{-1}\cdot b_2,a^{-1}\cdot c_2,i_2,j_2).
\end{align*} 
Hence, 
\begin{align*}
\alpha(b_2,c_2,i_2,j_2)&=\alpha(1,b_2^{-1}\cdot c_2,i_2,j_2)\quad \text{ (by (\ref{alphab2}), with } a=b_2\text{)}\\
&=\alpha(1,1,i_2,j_2)\quad \text{ (by (\ref{alphac2}), with } a=c_2^{-1}\cdot b_2\text{)},
\end{align*}
and similarly for $\beta(b_2,c_2,i_2,j_2)=\beta(1,1,i_2,j_2)$. 

Note also that
\begin{align*}
\lefteqn{(1,\Theta_{\varphi((b_1,c_1)K_{i_1,j_1})}(a_{\alpha(b_1,c_1,i_1,j_1)}))F((b_2,c_2)K_{i_2,j_2})}\\
&=(1,\Theta_{\varphi((b_1,c_1)K_{i_1,j})}(a_{\alpha(b_1,c_1,i_1,j)}))F((b_2,c_2)K_{i_2,j_2}),
\end{align*}
and that 
\begin{align*}
\lefteqn{(\Theta_{\varphi((b_1,c_1)K_{i_1,j_1})}(a_{\alpha(b_1,c_1,i_1,j_1)}),\Theta_{\varphi((b_1,c_1)K_{i_1,j_1})}(a_{\alpha(b_1,c_1,i_1,j_1)}))^{-1}F((b_2,c_2)K_{i_2,j_2})}\\
&=(\Theta_{\varphi((b_1,c_1)K_{i_1,j})}(a_{\alpha(b_1,c_1,i_1,j)}),\Theta_{\varphi((b_1,c_1)K_{i_1,j})}(a_{\alpha(b_1,c_1,i_1,j)}))^{-1}F((b_2,c_2)K_{i_2,j_2}),
\end{align*}
for all $b_1,b_2,c_1,c_2\in B$, $i_1,i_2\in I$,  $j_1,j\in J_{i_1}$ and
$j_2\in J_{i_2}$. Since 
$$\left\{a\in B : (1,a),(a,a)\in\bigcap_{i',j'}\bigcap_{b,c\in B}(b,c)L_{i',j'}(b,c)^{-1}\right\}=\{ 1\},$$ 
we have that
$$\Theta_{\varphi((b_1,c_1)K_{i_1,j_1})}(a_{\alpha(b_1,c_1,i_1,j_1)})=\Theta_{\varphi((b_1,c_1)K_{i_1,j})}(a_{\alpha(b_1,c_1,i_1,j)}),$$
for all $b_1,c_1\in B$, $i_1\in I$ and  $j_1,j\in J_{i_1}$. Therefore
$a_{\alpha(b_1,c_1,i_1,j_1)},a_{\alpha(b_1,c_1,i_1,j)}\in \alpha(1,1,i_1,k),$
for all $b_1,c_1\in B$, $i_1\in I$ and $j_1,j\in J_{i_1}$ and thus
$\alpha(b,c,i,j)=\alpha(1,1,i,k)$, for all $b,c\in B$, $i\in I_1$ and
$j,k\in J_{i}$. Moreover, $\Theta_{\varphi(K_{i,j})}=\Theta_{\varphi(K_{i,k})}$ for any $i\in I$ and 
any $j,k\in J_i$. 

For each $i\in I$ we choose an element $j_i\in
J_{i}$. Since $F$ is bijective, the map $\alpha: I\rightarrow I'$ defined by
$i\mapsto \alpha(i):=\alpha(1,1,i,j_i)$ is bijective and for each $i\in I$ the
map $\beta_i: J_i\rightarrow J'_{\alpha(1,1,i,j_i)}$ defined by $j\mapsto
\beta_i(j):=\beta(1,1,i,j)$ is bijective.

We shall see that there exists an
automorphism $\psi$ of the skew left brace $B$ such that
$$\psi(\Theta_{(b,c)}(a_i))=\Theta_{(\psi(b),\psi(c))\varphi(K_{i,j})}(a_{\alpha(i)})$$
and
$$\psi(K_{i,j})=\varphi(K_{i,j})L_{\alpha(i),\beta_i(j)}\varphi(K_{i,j})^{-1},$$
for all $b,c\in B$, $i\in I$ and $j\in J_i$. Let
$1=\Theta_{(b_1,c_1)}(a_{i_1})^{\varepsilon_1}\cdots
\Theta_{(b_m,c_m)}(a_{i_m})^{\varepsilon_m}$, for some $b_1,c_1,\dots ,b_m,c_m\in
B$, $i_1,\dots ,i_m\in I$ and $\varepsilon_1,\dots ,\varepsilon_m\in
\{ 1,-1\}$. By (\ref{F}), we have
\begin{align*}
F((b,c)K_{i,j})&=F((1,\Theta_{(b_1,c_1)}(a_{i_1})^{\varepsilon_1}\cdots
\Theta_{(b_m,c_m)}(a_{i_m})^{\varepsilon_m})(b,c)K_{i,j})\\
&=(1,\Theta_{\varphi((b_1,c_1)K_{i_1,j_{i_1}})}(a_{\alpha(i_1)})^{\varepsilon_1})\\
&\quad F((1,\Theta_{(b_2,c_2)}(a_{i_2})^{\varepsilon_2}\cdots
\Theta_{(b_m,c_m)}(a_{i_m})^{\varepsilon_m})(b,c)K_{i,j})\\
&=\left(1,\Theta_{\varphi((b_1,c_1)K_{i_1,j_{i_1}})}(a_{\alpha(i_1)})^{\varepsilon_1}\cdots
\Theta_{\varphi((b_m,c_m)K_{i_m,j_{i_m}})}(a_{\alpha(i_m)})^{\varepsilon_m}\right)\\
&\quad F((b,c)K_{i,j}),
\end{align*}
for all $b,c\in B$, $i\in I$ and $ j\in J_i$. 
By (\ref{F2}), in a similar way we obtain 
\begin{align*}
F((b,c)K_{i,j})&=\left(\Theta_{\varphi((b_1,c_1)K_{i_1,j_{i_1}})}(a_{\alpha(i_1)})^{\varepsilon_1}\cdots
\Theta_{\varphi((b_m,c_m)K_{i_m,j_{i_m}})}(a_{\alpha(i_m)})^{\varepsilon_m},\right.\\
&\left. \Theta_{\varphi((b_1,c_1)K_{i_1,j_{i_1}})}(a_{\alpha(i_1)})^{\varepsilon_1}\cdots
\Theta_{\varphi((b_m,c_m)K_{i_m,j_{i_m}})}(a_{\alpha(i_m)})^{\varepsilon_m}\right)F((b,c)K_{i,j})
\end{align*}
for all $b,c\in B$, $i\in I$ and $ j\in J_i$. 
Since
$(1,a),(a,a)\in\bigcap_{i',j'}\core_G(L_{i',j'})$ implies that $a=1$, we have that
$$\Theta_{\varphi((b_1,c_1)K_{i_1,j_{i_1}})}(a_{\alpha(i_1)})^{\varepsilon_1}\cdots
\Theta_{\varphi((b_m,c_m)K_{i_m,j_{i_m}})}(a_{\alpha(i_m)})^{\varepsilon_m}=1.$$
Therefore there exists a unique morphism $\psi\colon B\rightarrow B$
of multiplicative groups such that
$\psi(\Theta_{(b,c)}(a_i))=\Theta_{\varphi((b,c)K_{i,j})}(a_{\alpha(i)})$.
Since $Y$ generates the multiplicative group of $B$, by (\ref{F})
one can see that
$$\varphi((b,c)K_{i,j})L_{\alpha(i),\beta_i(j)}=F((b,c)K_{i,j})=(\psi(b),\psi(c))\varphi(K_{i,j})L_{\alpha(i),\beta_i(j)}.$$
Therefore, due to $L_{\alpha(i),\beta_i(j)}\subseteq
\St(a_{\alpha(i)})$, we have
$$\Theta_{\varphi((b,c)K_{i,j})}(a_{\alpha(i)})=\Theta_{(\psi(b),\psi(c))\varphi(K_{i,j})}(a_{\alpha(i)}).$$
Hence
$\psi(\Theta_{(b,c)}(a_i))=\Theta_{(\psi(b),\psi(c))\varphi(K_{i,j})}(a_{\alpha(i)})$.
From that, we obtain
\begin{eqnarray*}
\psi(b\star a_i)&=&\psi(b\lambda_{b^{-1}}(a_i))=\psi(b)\psi(\lambda_{b^{-1}}(a_i))\\
&=&\psi(b)\psi(\Theta_{(1,b^{-1})}(a_i))\\
&=&\psi(b)\Theta_{(1,\psi(b^{-1}))\varphi(K_{i,j})}(a_i)\\
&=&\psi(b)\lambda_{\psi(b)^{-1}}\Theta_{\varphi(K_{i,j})}(a_{\alpha(i)})\\
&=&\psi(b)\star\Theta_{\varphi(K_{i,j})}(a_{\alpha(i)})\\
&=&\psi(b)\star\psi(a_i).
\end{eqnarray*}
Now it is easy to see that $\psi$ is a morphism of braces.
Since $F$ is bijective and $$F((b,c)(b',c')K_{i,j})=(\psi(b),\psi(c))F((b',c')K_{i,j}),$$ it
follows that $\psi$ is bijective. Furthermore $(b,c)\in K_{i,j}$ if and
only if
\begin{eqnarray*}
\varphi(K_{i,j})L_{\alpha(i),\beta_i(j)}&=&F(K_{i,j})=F((b,c)K_{i,j})=(\psi(b),\psi(c))F(K_{i,j})\\
&=&(\psi(b),\psi(c))\varphi(K_{i,j})L_{\alpha(i),\beta_i(j)}.
\end{eqnarray*}
Therefore, defining $(y_{i,j},z_{i,j}):=\varphi(K_{i,j})$, $y_{i,j},z_{i,j}\in B$, the result follows.
\end{proof}

Summarizing, the last theorem says that two solutions constructed as
in Theorem~\ref{constSol} are isomorphic if we can find an automorphism
of the skew left brace $B$ that brings each $K_{i,j}$ to one $L_{i',j'}$,
taking into account that maybe the $L_{i',j'}$'s are permuted (that
is the reason for the $\alpha$ and $\beta_i$ maps), and that maybe
we have chosen another element of the orbit in the process (that is
the reason why the image $a_i$ is $\Theta_{(y_{i,j},z_{i,j})}(a_{\alpha(i)})$
and not just $a_{\alpha(i)}$, and it is also the reason why the
$L_{\alpha(i),\beta_i(j)}$ is conjugated by $(y_{i,j},z_{i,j})$).

As a corollary of the construction of solutions, we get the announced result characterizing 
skew left braces as permutation groups of solutions. 

\begin{corollary}
For any skew left brace $B$, there exist a non-degenerate solution $(X,r)$ such that 
$\mathcal{G}(X,r)\cong B$. Moreover, if $B$ is finite, we can choose $X$ to be finite. 
\end{corollary}
\begin{proof}
Take $I=\mathcal{O}$, $Y=B$, $J_i=\{1\}$ for any $i\in I$, and $K_{i,1}=0$ for any $i\in I$. 
This satisfies the conditions of Theorem~\ref{constSol}, so we obtain a non-degenerate 
solution $(X,r)$ such that $\mathcal{G}(X,r)\cong B$.
\end{proof}

\begin{remark}
In this paper, we have focused in the study of isomorphisms of solutions from the point of view of 
their permutation groups. 
Recently, D. Yang in \cite{DYang2} has presented some results relating isomorphic injective solutions and 
their structure groups. Recall that a solution $(X,r)$ is called {\em injective} if the natural map $i: X\to G(X,r)$ is 
injective. In terms 
of skew left braces, her main theorem can be stated in the following way:
\begin{theorem}[{\cite[Theorem 3.3]{DYang2}}]
Let $(X,r)$ and $(Y,s)$ be two injective non-degenerate solutions of the Yang-Baxter equation. 
Then $(X,r)$ is isomorphic to $(Y,s)$ if and only if there exists a brace isomorphism 
$\phi: G(X,r)\to G(Y,s)$ such that $\phi(X)=Y$. 
\end{theorem}
This theorem means that essentially the structure group with its brace structure is unique for each 
injective solution.

\end{remark}

In the next sections, we show how to use Theorems~\ref{constSol} and \ref{isomSol} to generate solutions of 
the Yang-Baxter equation with additional properties, like square-free, involutive or irretractable solutions. 
Note that there are two easy ways to find some subgroups of $(B,\star)\rtimes (B,\cdot)$: take subgroups of the form 
$(B,\star)\rtimes K$, where $K$ is a subgroup of $(B,\cdot)$, and those of the form $H\rtimes (B,\cdot)$, where 
$H$ is a subgroup of $(B,\star)$ invariant by the action of $(B,\cdot)$. The first kind of subgroups 
will appear in the involutive case, and the second one will appear in Section~\ref{sectRacks}.

\subsection{Square-free solutions}\label{SquareFree}

Recall that a non-degenerate solution $(X,r)$ is called {\em square-free} if $r(x,x)=(x,x)$. To construct this kind 
of solutions, we need a very special type of elements inside our brace. 

\begin{definition}
Let $B$ be a skew left brace. We say that $a\in B$ is {\em square-free} if $\lambda_a(a)=a$. 
\end{definition}

Note that, if $a$ is square-free, $a=\lambda_a(a)=\Theta_{(1,a)}(a)$, so $(1,a)$ belongs to $\St(a)$ in the notation of the last
section. Moreover, $\Theta_{(a,a)}(a)=a\star \lambda_a(a)\star a^\star=a$, hence also $(a,a)$ belongs to $\St(a)$.

Returning to $(X,r)$, for $r$ to be square-free we need first that all $x=(b,c)K_{i,j}$ satisfy
\begin{align*}
(b,c)K_{i,j}=f_x(x)=(1,\Theta_{(b,c)}(a_i))(b,c)K_{i,j};
\end{align*}
equivalently, $(1,\Theta_{(b,c)}(a_i))\in (b,c)K_{i,j}(b,c)^{-1}$. 
In particular, since 
$$(b,c)K_{i,j}(b,c)^{-1}\leq (b,c)\St(a_i)(b,c)^{-1}=\St(\Theta_{(b,c)}(a_i)),$$
all the elements $\Theta_{(b,c)}(a_i)$ of the orbit of $a_i$ must be square-free. 

Besides, we need that $g_x(x)=x$ for any $x$. Using the condition $(1,\Theta_{(b,c)}(a_i))\in\St(\Theta_{(b,c)}(a_i))$
obtained right above, we can simplify the expression for $g$:
\begin{align*}
(b,c)K_{i,j}=g_x(x)=(\Theta_{(b,c)}(a_i),\Theta_{(b,c)}(a_i))^{-1}(b,c)K_{i,j};
\end{align*}
equivalently, $(\Theta_{(b,c)}(a_i),\Theta_{(b,c)}(a_i))\in (b,c)K_{i,j}(b,c)^{-1}$. 

Therefore, to construct a square-free non-degenerate solution, we have the following proposition.
\begin{proposition}\label{constSquareFree}
Assume all the hypothesis of Theorem~\ref{constSol} plus the two conditions
\begin{align*}
&(1,\Theta_{(b,c)}(a_i))\in (b,c)K_{i,j}(b,c)^{-1},\text{ and }\\
&(\Theta_{(b,c)}(a_i),\Theta_{(b,c)}(a_i))\in (b,c)K_{i,j}(b,c)^{-1} \text{, for all }b,c\in B,~i\in I,~j\in J_i.
\end{align*}
Then, the constructed solution is square-free. 

Moreover, any square-free non-degenerate solution is constructed in this way. 
\end{proposition}
Note that, in particular, if a skew left brace does not contain a union of orbits of square-free elements 
generating all the brace, it can not produce square-free solutions.

\section{Involutive non-degenerate solutions and left braces}\label{sectInv}

\subsection{Connections with left braces}
In this section, we use our new results to study involutive solutions, which is the type 
of non-degenerate solutions that has received more attention recently. 
Recall that $r:X\times X\to X\times X$ is involutive if $r\circ r=\id$. With our usual notation for the components 
of a solution, $r$ is involutive if and only if $f_{f_x(y)}g_y(x)=x$ and $g_{g_y(x)}f_x(y)=y$, for any $x,y\in X$. 
Hence, if $r$ is non-degenerate, $f$ determines $g$ because $g_y(x)=f^{-1}_{f_x(y)}(x)$, and then an involutive 
non-degenerate solution is always of the form $r(x,y)=\left(f_x(y),~f^{-1}_{f_x(y)}(x)\right)$. 
Moreover, the presentation of the derived group becomes
\begin{align*}
A(X,r)&=\left\langle X\mid x\star f_x(y)=f_x(y)\star f_{f_x(y)}g_y(x)\right\rangle=
\left\langle X\mid x\star f_x(y)=f_x(y)\star x\right\rangle\\
&=\left\langle X\mid x\star y=y\star x\right\rangle,
\end{align*}
so it is isomorphic to the free abelian group generated by $X$. Then, 
the natural map $i_A: X\to A(X,r)=\Z^{(X)}$ is injective, and this implies that the map 
$i_G:X\to G(X,r)$ is also injective. In other words, two elements of $X$ do not become 
equal in $G(X,r)$. These facts will simplify 
some parts of the theory of skew left braces that we have developed in the last section.
For example, studying involutive non-degenerate solutions, we shall focus on a concrete class of skew 
left braces. 

\begin{definition}
Let $B$ be a skew left brace. We say that $B$ is a left brace if $(B,\star)$ is abelian. 
\end{definition}

In this case, we prefer to write the operation $\star$ additively $\star=+$, and to denote the 
common (multiplicative and additive) identity element by 0.
In this way, we recover the definition of left brace, introduced by Rump in \cite{Rump1}. 
We will use the equivalent definition of left brace in \cite{CJO}.

\begin{definition}
A {\em left brace} is a set $B$ with two binary operations, a sum $+$ and a multiplication $\cdot$, such that
$(B,+)$ is an abelian group, $(B,\cdot)$ is a group, and any $a,b,c\in B$ satisfies
$$
a\cdot (b+c)+a=a\cdot b+a\cdot c.
$$
A {\em right brace} is defined analogously, changing 
the last property by $(b+c)\cdot a+a=b\cdot a+c\cdot a$. A left brace also satisfying the condition of a 
right brace is called a {\em two-sided brace}. 
\end{definition}

Now the following classical facts about the relation between left braces and non-degenerate involutive solutions 
are a corollary of the relations between skew left braces and non-degenerate solutions. 

\begin{proposition}[{\cite[Theorem~1]{CJO}}]
Let $(X,r)$ be an involutive non-degenerate solution. Then, $G(X,r)$ and $\mathcal{G}(X,r)$ are 
left braces. Moreover, $(G(X,r),+)\cong \Z^{(X)}$, and 
$\mathcal{G}(X,r)\cong G(X,r)/\soc(G(X,r))\cong \left\langle f_x:x\in X\right\rangle\leq\sym_X.$
The sum in $\mathcal{G}(X,r)$ satisfies $f_x+f_y=f_x\circ f_{f^{-1}_x(y)}$ for all $x,y\in X$.
\end{proposition}
\begin{proof}
We know that $G(X,r)$ and $\mathcal{G}(X,r)$ are skew left braces, and
we have already observed that $(G(X,r),\star)=A(X,r)\cong \Z^{(X)}$, so $G(X,r)$ is a left brace. 
We know that $(\mathcal{G}(X,r),\star)$ is a quotient of $(G(X,r),\star)$, so this group is also abelian.

By the proof of Theorem~\ref{bracePerm}, $\mathcal{G}(X,r)$ is isomorphic to $G(X,r)/I$, where $I$ is the ideal 
$I=\{b\in G(X,r):f_b=\id,\stackrel{\sim}{g}_b=\id\}$, which is contained in the socle of 
$G(X,r)$. But in the involutive case $\soc(G(X,r))=\{b\in G(X,r):\lambda_b=\id\}=\{b\in G(X,r):f_b=\id\}$ 
(where in the last step it is used that $i:X\to G(X,r)$ is injective in the involutive case), and 
$\stackrel{\sim}{g}_b(x)=g_{\lambda^{-1}_x(b)}(x)=f^{-1}_b(x)$, so 
$\soc(G(X,r))=I$.  

Now consider the surjective morphism of groups 
$$\mathcal{G}(X,r)=
\left\langle \left(f_x,~\stackrel{\sim}{g}^{-1}_x\right):x\in X\right\rangle\to \left\langle f_x : x\in X\right\rangle,$$ which is 
given over the generators by $(f_x,\stackrel{\sim}{g}_x^{-1})\mapsto f_x$. This morphism is injective because, in an 
involutive non-degenerate solution, $\stackrel{\sim}{g}_x(y)=g_{f^{-1}_y(x)}(y)=f^{-1}_{f_yf^{-1}_y(x)}(y)=f^{-1}_x(y)$, 
so $\stackrel{\sim}{g}_x^{-1}$ is uniquely determined by $f_x$. 
Thus $\mathcal{G}(X,r)\cong \left\langle f_x : x\in X\right\rangle$. 
\end{proof}

\subsection{Construction of involutive non-degenerate solutions}\label{sectConstInv}

The method of construction of all the non-degenerate solutions 
associated to a fixed skew left brace can be used in 
particular for involutive non-degenerate solutions associated to 
a fixed left brace. The fundamental difference 
between the two cases is that, in a left brace, 
the action $\Theta:(B,+)\rtimes (B,\cdot)\to \aut(B,+)$ is reduced to 
$$\Theta_{(a,b)}(c)=a+ \lambda_b(c)- a=\lambda_b(c)$$ for any $a,b,c\in B$. In other words, in the involutive case, 
we may simply consider the action $\lambda: (B,\cdot)\to\aut(B,+)$; this modifies a bit the previous 
theorems about construction of solutions. 

A remark is in order here to understand completely the relation between the method of this section 
and the ones in Section~\ref{sectConstSol}: a non-degenerate involutive solution always 
gives rise to a structure of left brace over its permutation group. Hence if we are looking for 
involutive solutions, we must always work with left braces, and in fact in this section we are going to 
present a method to recover all the non-degenerate involutive solutions from left braces. But left braces 
also recover some non-involutive solutions with the method from Section~\ref{sectConstSol}, as the following example shows. 

\begin{example}\label{exLeftBraceNotInj}
Let $B$ be the brace with $(B,\star)=(\Z/(2),+)$, and trivial brace structure; i.e. $b\cdot c:=b+c$ for any $b,c\in \Z/(2)$. 
With the notation of Theorem~\ref{constSol}, in this case $G=(B,\star)\rtimes (B,\cdot)=\Z/(2)\times\Z/(2)$, and $\Theta$ is the trivial action: 
$\Theta_{(b,c)}(d)=d$ for any $b,c,d\in \Z/(2)$. Then, the decomposition of $B$ in $G$-orbits is $B=\{0\}\cup\{1\}$, and 
we take $Y=\{1\}\subseteq B$, $K_{1,1}=\{(0,0)\}\leq St(1)=G$. 
We enumerate the elements of the set $X=G/K_{1,1}\cong\Z/(2)\times\Z/(2)$ by $x_1=(0,0)$, $x_2=(0,1)$, $x_3=(1,0)$ and $x_4=(1,1)$. 
It is not difficult to show that $f_{(b,c)}(b',c')=(0,1)+(b',c')$ and $g_{(b,c)}(b',c')=(1,1)+(b',c')$ for any 
$b,b',c,c'\in B$. Hence $f_x=f=(x_1,x_2)(x_3,x_4)$, $g_x=g=(x_1,x_4)(x_2,x_3)\in\sym_X$
for any $x\in X$. This solution is not involutive because $g\neq f^{-1}$, but $\mathcal{G}(X,r)\cong B$, which is a left brace.  
Observe also that the map $i:X\to G(X,r)$ is not injective in this case. 
\end{example}

In other words, applying Theorem~\ref{constSol}, skew left braces $B$ with $(B,\star)$ a non-abelian group only recover non-involutive 
non-degenerate solutions. Left braces recover both involutive and non-involutive solutions, and next we show how to recover only the 
involutive non-degenerate solutions.  Our Theorems \ref{constInvSol} and \ref{isomInvSol} were first proved in \cite{BCJ}, 
but here our aim is to show how they fit in a natural way in the general picture.

We need first some notation, analogous to the notation of Section~\ref{sectConstSol}. 
The stabiliser of $a\in B$ by the action $\lambda$ is
denoted $\St(a)$. For $a\in B$, let $B_a=\{\lambda_b(a)\mid b\in
B\}$ be the orbit of $a$, and let $\mathcal{O}=\{ B_a\mid a\in B\}$
be the set of orbits of the action $\lambda$.  For each $i\in
\mathcal{O}$, choose an element $a_i\in i$. Let $I$ be a subset of
$\mathcal{O}$, such that $Y=\bigcup_{i\in I}i$ satisfies $B=\langle
Y\rangle_+$, the additive subgroup generated by $Y$. For each $i\in
I$, let $J_i$ be a non-empty set and  let $\{ K_{i,j} \}_{j\in J_i}$ be
a family of subgroups of $\St(a_i)$ such that
$$\bigcap_{i\in I}\bigcap_{j\in J_i}\bigcap_{b\in B}bK_{i,j}b^{-1}=\bigcap_{i\in I}\bigcap_{j\in J_i}\core_{(B,\cdot)}(K_{i,j})=\{ 1\}.$$
Note that if one of the subgroups $K_{i,j}$ is trivial, then this
last condition is satisfied.

The next result is the construction of all the involutive non-degenerate solutions.
It is proved as a consequence of the construction of non-degenerate solutions of Theorem~\ref{constSol}. 
An alternative proof can be found in \cite[Theorem~3.1]{BCJ}.

\begin{theorem}\label{constInvSol}
 With the above notation, let  $X:=\bigsqcup_{i\in I}\bigsqcup_{j\in
J_i}B/K_{i,j}$ be the disjoint union of the sets of left cosets
$B/K_{i,j}$. Then, $(X,r)$, where $r$ is the map
$$
\begin{array}{cccc}
r\colon& X\times X &\longrightarrow & X\times X\\
& (b_1K_{i_1,j_1},~b_2K_{i_2,j_2}) &\mapsto
&\left(f_{b_1K_{i_1,j_1}}(b_2K_{i_2,j_2}),~f^{-1}_{f_{b_1K_{i_1,j_1}}(b_2K_{i_2,j_2})}(b_1K_{i_1,j_1})\right),
\end{array}
$$
 with
$f_{b_1K_{i_1,j_1}}(b_2K_{i_2,j_2})=\lambda_{b_1}(a_{i_1})b_2K_{i_2,j_2}$,
is an involutive non-degenerate solution such that $\mathcal{G}(X,r)\cong B$ as left
braces.

Moreover, any non-degenerate involutive solution $(Z,t)$, with $\mathcal{G}(Z,t)\cong B$ as
left braces, is isomorphic to  such a solution.
\end{theorem}
\begin{proof}
Theorem~\ref{constSol} describes precisely the non-degenerate solutions $(X,r)$ such that $\mathcal{G}(X,r)$ is isomorphic 
to $B$ as left braces. 
Recall that, in that theorem, we consider the group $G=(B,+)\rtimes (B,\cdot)$, where the action of $(B,\cdot)$ 
over $(B,+)$ is $\lambda$, and we also consider the action $\Theta: G\to Aut(B,+)$. 
We have already observed that, when 
$B$ is a left brace, $\Theta_{(a,b)}=\lambda_b$. 
The stabiliser of $a\in B$ by the action $\Theta$ (in our case, the 
action $\lambda$) is
denoted $\St(a)$. For $a\in B$, let $B_a=\{\lambda_{c}(a)\mid c\in
B\}$ be the orbit of $a$, and let $\mathcal{O}=\{ B_a\mid a\in B\}$
be the set of orbits of the action $\lambda$.  For each $i\in
\mathcal{O}$, choose an element $a_i\in i$. Then, Theorem~\ref{constSol} tells us that 
any non-degenerate solution is constructed choosing a subset $I$ of
$\mathcal{O}$ such that $Y=\bigcup_{i\in I}i$ satisfies $B=\langle
Y\rangle_+$, and choosing for each $i\in
I$ a non-empty set $J_i$ and a family of subgroups $\{ \overline{K}_{i,j} \}_{j\in J_i}$ of $\St(a_i)$ such that
$$\left\{a\in B:(1,a),(a,a)\in\bigcap_{i\in I}\bigcap_{j\in J_i}\core_G(\overline{K}_{i,j})\right\}=\{ 1\}.$$

Now, to prove our theorem, we are going to show that a non-degenerate solution constructed with Theorem~\ref{constSol} is involutive 
if and only if any $\overline{K}_{i,j}$ is of the form 
$$\overline{K}_{i,j}=\{(a,b)\in G: a\in B, b\in K_{i,j}\}$$
for some subgroup $K_{i,j}$ of $(B,\cdot)$. If each $\overline{K}_{i,j}$ is of this form, then $G/\overline{K}_{i,j}\cong B/K_{i,j}$
as $G$-sets. Moreover, for any $x=(b_1,c_1)\overline{K}_{i_1,j_1}$ and $y=(b_2,c_2)\overline{K}_{i_2,j_2}$, the $f$ and $g$ defined in Theorem~\ref{constSol} are simplified as
$$
f_x(y)=(1,\Theta_{(b_1,c_1)}(a_{i_1}))(b_2,c_2)\overline{K}_{i_2,j_2}=(1,\lambda_{c_1}(a_{i_1})\cdot c_2)\overline{K}_{i_2,j_2},
$$
and
\begin{align*}
g_y(x)&=(\Theta_{(1,\Theta_{(b_1,c_1)}(a_{i_1}))}\Theta_{(b_2,c_2)}(a_{i_2}),~
\Theta_{(1,\Theta_{(b_1,c_1)}(a_{i_1}))}\Theta_{(b_2,c_2)}(a_{i_2}))^{-1}(b_1,c_1)\overline{K}_{i_1,j_1}\\
&=(1,\lambda_{\lambda_{c_1}(a_{i_1})}\lambda_{c_2}(a_{i_2}))^{-1}(1,c_1)\overline{K}_{i_1,j_1}=
f^{-1}_{f_x(y)}(x).
\end{align*}
Thus $(X,r)$ is involutive because $g_y(x)=f^{-1}_{f_x(y)}(x)$ for any $x,y\in X$. Note that, after the identification 
 $G/\overline{K}_{i,j}\cong B/K_{i,j}$, these are the $f$ and $g$ maps from the statement of our theorem. 

Conversely, if $(X,r)$ is involutive, suppose 
to arrive to a contradiction that for some $\overline{K}_{i,j}$ there exists $(b,1)\in G\setminus \overline{K}_{i,j}$. Then, since 
$\langle Y\rangle_+=B$, $(b,1)$ decomposes as a product of elements of $\{(y,1): y\in Y\}$. Some term in this decomposition
does not belong to $\overline{K}_{i,j}$ (if all of them belong to $\overline{K}_{i,j}$, then $(b,1)$ belongs to it also), 
so there exists a $z\in Y$ such that $(z,1)\not\in \overline{K}_{i,j}$. 
In other words, $z\in Y$ satisfies $(z,1)\overline{K}_{i,j}\neq \overline{K}_{i,j}$. 

Now take $b_1=c_1=\lambda_{-z}^{-1}(z)$, and $x=(b_1,c_1)\overline{K}_{i,j}$. Moreover, since $z\in Y$, 
let $i_2$ be the orbit of $z$, and take $c_2$ such that 
$\lambda_{-z}\lambda_{\lambda_{c_1}(a_{i})}\lambda_{c_2}(a_{i_2})=z$. 
Take $b_2=1$ and $y=(b_2,c_2)\overline{K}_{i_2,j_2}$. Then, on one side,
\begin{align*}
f^{-1}_{f_{x}(y)}(x)&=(1,\lambda_{\lambda_{c_1}(a_{i})}\lambda_{c_2}(a_{i_2}))^{-1} (b_1,c_1)\overline{K}_{i,j}
=(1,\lambda^{-1}_{-z}(z))^{-1}(\lambda^{-1}_{-z}(z),~\lambda^{-1}_{-z}(z))\overline{K}_{i,j}\\
&=\left(\lambda^{-1}_{\lambda^{-1}_{-z}(z)}\lambda^{-1}_{-z}(z),1\right)\overline{K}_{i,j}
=(\lambda^{-1}_{-z+z}(z),1)\overline{K}_{i,j}\\
&=(z,1)\overline{K}_{i,j},
\end{align*}
and, on the other,
\begin{align*}
g_y(x)&=
(\lambda_{\lambda_{c_1}(a_{i})}\lambda_{c_2}(a_{i_2}),~\lambda_{\lambda_{c_1}(a_{i})}\lambda_{c_2}(a_{i_2}))^{-1}(b_1,c_1)\overline{K}_{i,j_1}\\
&=(\lambda^{-1}_{-z}(z),~\lambda^{-1}_{-z}(z))^{-1}(\lambda^{-1}_{-z}(z),~\lambda^{-1}_{-z}(z))\overline{K}_{i,j}\\
&=(1,1)\overline{K}_{i,j}=\overline{K}_{i_1,j_1}.
\end{align*}
Thus $(X,r)$ is not involutive because $g_y(x)\neq f^{-1}_{f_x(y)}(x)$, which is a contradiction.

It is straightforward to see that this is enough to prove the theorem. Observe that the condition 
$$\left\{a\in B:(1,a),(a,a)\in\bigcap_{i\in I}\bigcap_{j\in J_i}\bigcap_{b,c\in B}(b,c)\overline{K}_{i,j}(b,c)^{-1}\right\}=\{ 1\}$$
of Theorem~\ref{constSol} becomes equivalent to the condition
$$\bigcap_{i\in I}\bigcap_{j\in J_i}\bigcap_{b\in B}bK_{i,j}b^{-1}=\bigcap_{i\in I}\bigcap_{j\in J_i}\core_{(B,\cdot)}(K_{i,j})=\{ 1\}$$
of our theorem.
\end{proof}

Using the last theorem, we are able to construct all the non-degenerate involutive solutions 
from a given left brace, but it might happen that there are repetitions, that some of them 
are isomorphic. Now we want to explain the isomorphism of solutions in terms of automorphisms 
of the left brace. Again, this will be a restriction of Theorem~\ref{isomSol} to the case of 
left braces. 

Let $B$ be a left brace and let $\mathcal{O}$, $I$, $a_i$, $J_i$,
$K_{i,j}$ be as in Theorem~\ref{constInvSol}. Let $(X,r)$ be
the non-degenerate involutive solution of the statement of Theorem~\ref{constInvSol}.
Let $I'\subseteq \mathcal{O}$ such that $Y'=\bigcup_{i'\in I'}i'$
satisfy $B=\langle Y'\rangle_+$.
For each $i'\in I'$, let $\{ L_{i',j'} \}_{j'\in J'_{i'}}$ be a
non-empty family of subgroups of $\St(a_{i'})$ such that
$$\bigcap_{i'\in I'}\bigcap_{j'\in J'_{i'}}\core_{(B,\cdot)}(L_{i',j'})=\{ 1\}.$$
Let $(X',r')$ be the corresponding solution defined as in
the statement of Theorem~\ref{constInvSol}, that is
\begin{eqnarray*}
&&r'(b_1L_{i'_1,j'_1},b_2L_{i'_2,j'_2})=(\lambda_{b_1}(a_{i'_1})b_2L_{i'_2,j'_2},\lambda_{\lambda_{b_1}(a_{i'_1})b_2}(a_{i'_2})^{-1}b_1L_{i'_1,j'_1}).
\end{eqnarray*}
We shall characterize when $(X,r)$ and $(X',r')$ are isomorphic in
the following result. Since a left brace is the same as a skew left brace 
with abelian star group, it is easy to prove the following result using 
Theorem~\ref{isomSol}.

\begin{theorem}\label{isomInvSol}
The solutions $(X,r)$ and $(X',r')$ are isomorphic if and only if
there exist an automorphism $\psi$ of the left brace $B$, a
bijective map $\alpha\colon I\rightarrow I'$, a bijective map
$\beta_i\colon J_i\rightarrow J'_{\alpha(i)}$ and $z_{i,j}\in B$,
for each $i\in I$ and $j\in J_i$, such that
$$\psi(a_i)=\lambda_{z_{i,j}}(a_{\alpha(i)})\quad\mbox{and}\quad \psi(K_{i,j})=z_{i,j}L_{\alpha(i),\beta_i(j)}z_{i,j}^{-1},$$
for all $i\in I_1$ and $j\in J_i$.
\end{theorem}

An alternative proof of the last theorem can be found in \cite[Theorem~4.1]{BCJ}.

As a corollary, we obtain another of the classical results about non-degenerate involutive solutions:
any left brace is the permutation group of a solution, so left braces and permutation groups are equivalent 
concepts. 

\begin{corollary}[{\cite[Theorem~2.1]{CJR}}]
For any left brace $B$, there exist a non-degenerate involutive solution $(X,r)$ such that 
$\mathcal{G}(X,r)\cong B$. Moreover, if $B$ is finite, we can choose $X$ to be finite. 
\end{corollary}
\begin{proof}
Take $I=\mathcal{O}$, $Y=B$, $J_i=\{1\}$ for any $i\in I$, and $K_{i,1}=0$ for any $i\in I$. 
This satisfies the conditions of Theorem~\ref{constInvSol}, so we obtain a non-degenerate 
involutive solution $(X,r)$ such that $\mathcal{G}(X,r)\cong B$.
\end{proof}

\subsection{Irretractable involutive solutions}

Let $(X,r)$ be a non-degenerate involutive solution. 
Recall that Etingof, Schedler and Soloviev defined in \cite{ESS} the following equivalence 
relation on $X$: $x\backsim y$ iff $f_x=f_y$. Then, the retraction of $(X,r)$, denoted by $\ret(X,r)$, is the 
solution defined over $X/\backsim$ in the natural way. We say that $(X,r)$ is {\em irretractable} 
if $\ret(X,r)=(X,r)$. 

In this section, we are going to show how to construct all the irretractable solutions from a given left brace. Due to 
\cite[Lemma 2.1 and Remark 2.2]{BCJO}, left braces that generate irretractable solutions are precisely the ones 
with trivial socle. 

Before the theorem, note that $\soc(G(X,r))=\bigcap \core_{(B,\cdot)}(\St(a_i))$, where the intersection is over all the representatives 
$a_i$ of the orbits. Hence $\soc(G(X,r))\subseteq \bigcap_{i\in I}\core_{(B,\cdot)}(\St(a_i))$.

\begin{theorem}\label{constIrret}
Let $B$ be a left brace with trivial socle. Using the notation of Theorem~\ref{constInvSol}, assume that
$J_i=\{1\}$ for any $i\in I$ and that $K_{i,1}=\St(a_i)$ for any $i\in I$. Then, the constructed solution 
is irretractable. 

Moreover, any irretractable non-degenerate involutive solution can be constructed in this way. 
\end{theorem}
\begin{proof}
To prove the first statement, assume that $f_x=f_y$ for some $x=b_1K_{i_1,1}$ and $y=b_2K_{i_2,1}$ in $X$ 
(note that we will write always $j=1$ because $J_i=\{1\}$ by assumption). 
By definition of $f$, this means that $\lambda_{b_1}(a_{i_1})b K_{i,1}=\lambda_{b_2}(a_{i_2})b K_{i,1}$ for any $b\in B$, and any $i$. 
In this case, the fact that the 
intersection $$\bigcap_{i,j} \core_{(B,\cdot)}(K_{i,j})=\bigcap_{i,j}\bigcap_{b\in B} bK_{i,j}b^{-1}$$ is equal to $\{0\}$ implies that 
$\lambda_{b_1}(a_{i_1})=\lambda_{b_2}(a_{i_2})$. This last equality is impossible for elements in different orbits, so $i_1=i_2$. 
Hence $b_1^{-1} b_2$ belongs to the stabilizer of $a_{i_1}$. By assumption, $K_{i_1,1}$ is the unique $K_{i,1}$ equal to $\St(a_{i_1})$, so 
$b_1^{-1} b_2\in K_{i_1,1}.$ In other words, $x=b_1K_{i_1,1}=b_2K_{i_1,1}=y$, as we wanted to prove. 

For the converse statement, assume first that $K_{i,j}\lneq \St(a_i)$ for some $i,j$. In that case, 
we can choose $x\in \St(a_i)\setminus K_{i,j}$. Then,
$$
f_{xK_{i,j}}(bK_{i_2,j_2})=\lambda_x(a_i)bK_{i_2,j_2}=a_ibK_{i_2,j_2},
$$
and
$$
f_{0K_{i,j}}(bK_{i_2,j_2})=\lambda_0(a_i)bK_{i_2,j_2}=a_ibK_{i_2,j_2}.
$$
Thus $f_{xK_{i,j}}=f_{0K_{i,j}}$ for $xK_{i,j}\neq K_{i,j}=0K_{i,j}$, and the solution is not irretractable. 

On the other hand, assume that $\ord{J_{i}}\neq 1$ for some $i\in I$. Then, by the reasoning in the previous paragraph, 
$K_{i,j}=\St(a_i)$ for any $j\in J_i$. Hence there are two different orbits $X_{i_1}$ and $X_{i_2}$ in $X$ 
of the form $B/\St(a_i)$. But then $x=0\St(a_i)\in X_{i_1}$ and $y=0\St(a_i)\in X_{i_2}$ satisfy $f_x=f_y$, so 
the solution is not irretractable. 
\end{proof}

Recall that, given a left brace $B$, the {\em associated solution} to $B$ is defined in \cite[Lemma~2]{CJO} as 
$$
\begin{array}{cccc}
r_B:& B\times B& \longrightarrow & B\times B\\
   & (a,b)& \longmapsto & \left(\lambda_a(b),~\lambda^{-1}_{\lambda_a(b)}(a)\right).
\end{array}
$$
This solution correspond to the case $I=\mathcal{O}$, $Y=B$, $J_i=\{1\}$, and $K_{i,1}=\St(a_i)$ for any $i\in I$. 

Since in general $\ret(B,r_b)=B/\soc(B)$, this solution is irretractable when $B$ has trivial socle.
All this means that the maximal irretractable solution is the associated solution to $B$, because all the others 
are sub-solutions of that one (we can obtain them by eareasing some orbits from $(B,r_B)$, being careful of course with 
the properties of Theorem~\ref{constInvSol}).

\begin{example}
As an example, we show now how to construct a recent example of solution due to Vendramin \cite{VendraminCounter}. 
It is an example of a non-degenerate involutive solution of size $8$ which is both irretractable and square-free. 
First of all, consider the abelian group $(\Z/(2))^6$ with the following lambda map:
$$
\lambda_{(y_1,\dots,y_6)}:=
\begin{pmatrix}
1& y_3& B& 0& 0& 0\\
0& 1& A& 0& 0& 0\\
0& 0& 1& 0& 0& 0\\
0& 0& 0& 1& y_6& y_5+y_6(y_1+y_2+y_2y_3)\\
0& 0& 0& 0& 1& y_1+y_2+y_2y_3\\
0& 0& 0& 0& 0& 1
\end{pmatrix},
$$
where $A=A(y_1,\dots,y_6):=(y_4+y_5+y_5y_6)+y_6(y_1+y_2+y_2y_3)$ and $B=B(y_1,\dots,y_6):=y_2+y_3\cdot A$.
We know that its multiplicative group is isomorphic to $D_4\times D_4$.

It is straightforward to check (for example, using \cite[Lemma~2.6]{B3}) that this defines a left brace, and, moreover, that 
its socle is trivial. The two sets
$$
\mathcal{O}_1=\mathcal{O}_{(0,0,1,0,0,0)}=\{(y_1,y_2,1,0,0,0):y_i\in\Z/(2)\},
$$
$$
\mathcal{O}_2=\mathcal{O}_{(0,0,0,0,0,1)}=\{(0,0,0,y_1,y_2,1):y_i\in\Z/(2)\},
$$
are orbits of this action, whose union generates all $B$. The corresponding stabilizers are
$$
K_{1,1}=\St(0,0,1,0,0,0)=\{(y_1,0,y_3,y_4,y_5,y_6):y_4+y_5+y_6+y_1y_6=0\},
$$
$$
K_{2,1}=\St(0,0,0,0,0,1)=\{(y_1,0,y_3,y_4,0,y_6):y_4=y_1y_6\},
$$
whose intersection is $\{(0,0,y_3,0,0,y_6):y_3,y_6\in\Z/(2)\}$. The core of this last subgroup is zero. 
Hence, by Theorem~\ref{constIrret}, we obtain an irretractable solution. 

Moreover, note that $(0,0,1,0,0,0)\in\St(0,0,1,0,0,0)$, and, more generally,
$$\lambda_b(0,0,1,0,0,0)\in \St(\lambda_b(0,0,1,0,0,0))=bK_{1,1}b^{-1}$$ for any $b\in B$. 
The analogous result holds for $(0,0,0,0,0,1)$. 
Thus, by Proposition~\ref{constSquareFree}, we obtain a square-free solution. 

To recover Vendramin's solution with the notation of \cite[Example~3.9]{VendraminCounter}, use the following identification:
$$
x_1=(0,0,1,0,0,0),~x_3=(1,0,1,0,0,0),~x_5=(0,1,1,0,0,0),~x_7=(1,1,1,0,0,0),
$$
$$
x_2=(0,0,0,0,0,1),~x_4=(0,0,0,1,0,1),~x_6=(0,0,0,0,1,1),~x_8=(0,0,0,1,1,1).
$$

\end{example}

\section{Construction of combinatorial invariants of knot theory: racks and quandles}\label{sectRacks}

In this section, we relate our set-theoretic solutions with some classes of algebraic structures that 
have become recently very important in knot theory, producing some new combinatorial knot invariants, 
and appearing in virtual knot theory, a generalization of knot theory 
(see \cite{CJKLS} and \cite{Nelson}, and the references 
there). Besides, racks have become important 
for Hopf algebras, due to their relations with Nichols algebras and finite dimensional 
pointed Hopf algebras, the classification of 
which have been intensively studied in the last years, see \cite{AG}. 

\begin{definition}
A \emph{shelf} is a set $X$ with a (left) self-distributive operation $\circ$; i.e. it satisfies 
$a\circ (b\circ c)=(a\circ b)\circ (a\circ c)$ for every $a,b,c\in X$.
A \emph{rack} is a shelf $(X,\circ)$ such that the maps $b\mapsto a\circ b$ are 
bijective for any $a\in X$. A \emph{quandle} is a rack $(X,\circ)$ such that $a\circ a=a$ for any $a\in X$. 

A \emph{birack} is a set $X$ with two operations $\circ$ and $\star$ such that 
the maps $x\mapsto z\circ x$, and $x\mapsto z\star x$ are bijective for any $z\in X$, and 
the following identities are satisfied:
\begin{align*}
x\circ (y\circ z)&=(x\circ y)\circ ((y\star x)\circ z),\\
x\star (y\star z)&=(x\star y)\star ((y\circ x)\star z),\text{ and}\\
((x\star y)\circ z)\star (y\circ x)&=((x\circ z)\star y)\circ (z\star x). 
\end{align*}
\end{definition}

Note that $(X,\circ)$ is a shelf if and only if $r(x,y)=(y,y\circ x)$ is a set-theoretic solution 
of the Yang-Baxter equation, that $(X,\circ)$ is a rack if and only if $r(x,y)=(y,y\circ x)$ is 
a non-degenerate solution, and that $(X,\circ)$ is a quandle if and only if $r(x,y)=(y,y\circ x)$
is a square-free non-degenerate solution. Moreover, $(X,\circ,\star)$ is a birack if and only if 
$r(x,y)=(x\circ y,y\star x)$ is a non-degenerate solution. In this way, the problem of constructing new 
examples of these algebraic structures is equivalent to our problem of finding new solutions (in the case of 
racks and quandles, new solutions with trivial first component). 
In fact, with Theorems~\ref{constSol} and \ref{isomSol}, these problems can be studied through brace theory.

Finally, the following result shows that the general non-degenerate solutions are important in rack theory, since 
we can associated to any such solution a solution with first component equal to the identity (in other words, a rack). 
Observe that for this application we need non-involutive solutions because, for involutive solutions, the associated rack is 
trivial. 

\begin{proposition}[{\cite[Theorem~2.3]{Soloviev}}]
Let $(X,r)$ be a non-degenerate solution. Then, the operation $\circ$ over $X$ given 
by $y\circ x:=f_y g_{f^{-1}_x(y)}(x)$ defines a rack. 
The structure group of the non-degenerate solution $(X,s)$, where $s(x,y)=(y,y\circ x)$, is isomorphic to $A(X,r)$. 
\end{proposition}

\subsection{Construction of racks and quandles}

When we restrict ourselves to the case of racks, the results can be improved. Observe that,  
if $s(x,y)=(f_x(y),g_y(x))=(y,y\circ x)$ is the solution corresponding to a rack $(X,\circ)$, 
then the structure group of $(X,s)$ has a presentation of the form 
$$
G(X,s)=\left\langle X\mid x\cdot y=y\cdot (y\circ x)\right\rangle,
$$
and the derived group has a presentation of the form
$$
A(X,s)=\left\langle X\mid y^{\star}\star x\star y=y\circ x\right\rangle.
$$
Moreover, the two operations coincide because $x\cdot y=x\star f_x(y)=x\star y$. Therefore, in this case, 
the structure group of $(X,s)$ is a trivial brace. This implies that the permutation group, 
which in this case is equal to 
$$\mathcal{G}(X,r)=\langle (\id,g^{-1}_x):x\in X\rangle\cong \langle g^{-1}_x:x\in X\rangle,$$
is also a trivial brace. 
In other words, 
racks can be studied using only group theory, and forgetting about skew left braces. 

Nevertheless, our results of Section \ref{sectNonDegSol} can be applied here in a useful way. 
The next result shows how to recover all the racks from the group structure of $\mathcal{G}(X,r)$. 
We use the notation $C_G(g)$ for the centralizer of $g$ in $G$.

\begin{theorem}
Let $G$ be a group. Let $\{C_i\}_{i\in I}$ be a set of different conjugacy classes of $G$, and let 
$Y=\bigcup_{i\in I} C_i$. 
Fix an element $g_i\in C_i$ for each $i\in I$, and consider subgroups $K_{i,j}\leq C_G(g_i)$, $j\in J_i$. 
Suppose that $\langle Y\rangle=G$ and that $\core(\bigcap_{i,j} K_{i,j})=\{1\}$. 
Then, the set 
$X=\bigsqcup_{i,j} G/K_{i,j}$, with operation $(xK_{i,j})\circ (yK_{a,b})=xg_{i}^{-1}x^{-1}yK_{a,b}$, is a rack. 

Moreover, any rack is of this form. 
\end{theorem}
\begin{proof}
Consider $G$ as a trivial brace. In this case, $(G,\star)\rtimes (G,\cdot)=G\times G$. The action 
$\Theta$ in this case is $\Theta_{(a,b)}(c)=a\star \lambda_b(c)\star a^\star=a\cdot c\cdot a^{-1}$, so 
it is conjugation by the first component, and therefore its orbits coincide with the conjugacy classes of $G$. 
The stabilizers of these orbits are of the form $C_G(g)\times G$. Hence if we take subgroups of the form 
$K_{i,j}\times G$, $i\in I$, $j\in J_i$, we have 
$(G\times G)/(K_{i,j}\times G)\cong G/K_{i,j}$ as $G\times G$-sets, and
all the hypothesis of Theorem~\ref{constSol} are satisfied, so we construct a non-degenerate solution of the Yang-Baxter equation. 
Finally, to check that this solution is a rack, it is enough to 
compute the maps $f$ and $g$ in this case: for $a=(x,z)(K_{i_1,j_1}\times G)=(x,1)(K_{i_1,j_1}\times G)$
and $b=(y,1)(K_{i_2,j_2}\times G)$, we have
\begin{align*}
f_a(b)=(1,x\cdot g_{i_1}\cdot x^{-1})(y,1)(K_{i_2,j_2}\times G)=(y,1)(K_{i_2,j_2}\times G)=b,
\end{align*} 
\begin{align*}
g_b(a)=(\Theta_{(y,1)}(g_{i_2}),1)^{-1}(x,1)(K_{i_1,j_1}\times G)=(y g_{i_2}^{-1} y^{-1} x,1)(K_{i_1,j_1}\times G).
\end{align*} 
The fact that the map $f_a$ in the first component of the solution is the identity shows that it is a rack. Moreover,
 the second component $g$ coincides with the operation $\circ$ defined in the statement. 

Conversely, if $(X,\circ)$ is any rack, we have noted that $B=\mathcal{G}(X,r)$ is a trivial brace, so again
$(B,\star)\rtimes (B,\cdot)=B\times B$, and the action $\Theta$ is conjugation by the first component. 
Since $(X,r)$, where $r(x,y)=(y,y\circ x)$ is a non-degenerate solution, it can be constructed as in 
Theorem~\ref{constSol}. We are done if we are able to prove that in this case the subgroups of $B\times B$ appearing in that theorem 
are of the form $K\times B$, where $K$ is a subgroup of $C_G(g)$ for some $g\in B$. But this is easy to check using that 
$f_x$ must be the identity for any $x\in X$ because $(X,r)$ is the solution associated to a rack. 
\end{proof}

\begin{corollary}
With the notation of the last theorem, if $g_i$ belongs to $K_{i,j}$ for any $i\in I$ and any $j\in J_i$, then 
the rack $(X,\circ)$ constructed above is a quandle. Moreover, any quandle is constructed in this way. 
\end{corollary}
\begin{proof}
It is enough to observe that $(xK_{i,j})\circ (xK_{i,j})=xg_{i}^{-1}x^{-1}xK_{i,j}=xg_i^{-1} K_{i,j}$, and this is 
equal to $x K_{i,j}$ if and only if $g_i$ belongs to $K_{i,j}$.
\end{proof}

\begin{example}
Let $G$ be a finite simple group. It is true that $G$ is simple if and only if for any $x\in G$, $x\neq 1$, 
the conjugacy class of $x$ in $G$ generates the whole group $G$. Moreover, since the core of a subgroup 
is a normal subgroup of $G$, it is always trivial. So in this case any conjugacy class 
with a subgroup of the centralizer of any of its elements generates a rack. 
\end{example}

\section{An application to Hopf-Galois extensions}\label{sectHopfGalois}
In this final section, we want to point out an equivalence 
between finite skew left braces, and Hopf-Galois extensions, a class of Hopf 
algebras that generalizes Galois field extensions.

Given a group $A$, we define the holomorph of $A$ as $\hol(A):=A\rtimes \aut(A)$. 
When $A$ is equal to an elementary abelian $p$-group (i.e. $A\cong (\Z/(p))^n$), $\hol(A)$ is also known as 
the affine group of $A$. Regular subgroups of the affine group have been studied for example in \cite{LPS,Hegedus}, mainly
as subgroups of primitive permutation groups.  

 The following result is an easy generalization of \cite[Theorem 1]{CR}, which gives an equivalence between skew left braces 
and regular subgroups of the holomorph. 
Recall that a regular subgroup
of $\hol(A)$ is a subgroup $H\leq \hol(A)$ such that for any $w\in A$ there exists a unique 
$(v,M)\in H$ such that $(v,M)(w):=v\star M(w)=1$.

\begin{proposition}[{\cite[Theorem~4.2 and Proposition~4.3]{GV}}]\label{regular}
Let $(A,\star)$ be a group.
\begin{enumerate}[(1)]
\item Let $B$ be a skew left brace with star group isomorphic to $A$. Then, $\{(a,\lambda_a):a\in A\}$ is a 
regular subgroup of $\hol(A)$.

Conversely, if $H$ is a regular subgroup of $\hol(A)$, we have $\pi_1(H)=A$, and 
the group $(A,\star)$ with the product
$$a\cdot b:=a\star\pi_2((\left.\pi_1\right|_H)^{-1}(a))(b)$$
is a skew left brace with multiplicative group isomorphic to $H$, where $\pi_1:H\to A$ and $\pi_2:H\to \aut(A)$ are 
the natural projections.

\item This defines a bijective correspondence between skew left braces with star group $A$, and regular
subgroups of $\hol(A)$. Moreover, 
isomorphic braces correspond to conjugate subgroups of $\hol(A)$ by elements of $\aut(A)$.
\end{enumerate}
\end{proposition}

The study of the Yang-Baxter equation has been intimately related to the theory of Hopf algebras 
since the pioneering work of Drinfeld \cite{Drinfeld2}. Hence it comes as no surprise that 
skew left braces, whose initial motivation was also the Yang-Baxter equation, 
have some connections with Hopf algebras too. Some of them are their relation with triangular 
semisimple and cosemisimple Hopf algebras explained in \cite{EtingofGelaki,BDG}, and the relation with 
finite dimensional pointed Hopf algebras through rack theory, explained in \cite{AG} (we clarify the 
connection with rack theory in Section~\ref{sectRacks}). These two classes of Hopf algebras 
have received a lot of attention recently since they are important in the program sketched in \cite[pages 376, 377]{AS} 
to try to obtain a classification of finite-dimensional Hopf algebras. 

In this section, we explain a new connection of skew left braces with another topic of recent 
interest in the study of Hopf algebras: Hopf-Galois extensions. In fact, we are going to show that the 
two algebraic structures are equivalent, using the equivalence of braces with regular subgroups of the holomorph. 

Let $L/K$ be a finite 
field extension. We say that $L/K$ is a Hopf-Galois extension if there exists a Hopf algebra $H$ over $K$
of finite dimension, and $\mu:H\to \operatorname{End}_K(L)$ a Hopf action such that 
$
(1,\mu): L\otimes_K H\to \operatorname{End}_K(L)
$
is an isomorphism of $K$-vector spaces, where $(1,\mu)(l\otimes h)(t)=l\cdot (\mu(h)(t))$. 
For example, when $L/K$ is a Galois extension with Galois group $G$, the Hopf algebra
$H=K[G]$ satisfies these properties. 

It is proved in \cite{By} that, when $L/K$ is a finite Galois extension and $G=Gal(L/K)$, 
if $G'$ is a group such that there exists an injective morphism of groups $\gamma: G\hookrightarrow\hol(G')=G'\rtimes\aut(G')$
satisfying that $\gamma(G)$ is a regular subgroup of $\hol(G')$, then the Hopf algebra $H=(L[G'])^G$ 
of $G$-fixed points in $L[G']$
defines a Hopf-Galois extension of 
$L/K$. Moreover, any Hopf-Galois extension in this case is of this form. Hence this translates the problem 
of finding Hopf-Galois extensions completely in group-theoretical terms: given a Galois group $G$, first find 
all the regular subgroups of $\hol(G')$ isomorphic to $G$, and second, find all the injective morphisms from $G$ 
to $\hol(G')$ with one of these subgroups as image. Observe that the regularity property implies that $\ord{G}=\ord{G'}$.

Note that, by Proposition \ref{regular}, to find regular subgroups of $\hol(G')$ is equivalent to 
find skew left braces with star group isomorphic to $G'$. 
So the first part of the problem of construction of Hopf-Galois extensions
of a Galois extension $L/K$ with Galois group equal to $G$ is equivalent to 
our problem of construction of skew left braces with multiplicative group isomorphic to $G$. 
We hope that this connection between these two theories would be fruitful in the future.  As an example, 
we translate some of the results about Hopf-Galois extensions to our setting:

\begin{enumerate}[(a)]
\item Skew left braces of order $p\cdot q$, where $p$ and $q$ are two different primes,
are completely classified (Byott \cite{Byott1}). 

\item Skew left braces with multiplicative group equal to a finite simple group 
are completely classified. In fact, for any simple group $G$, there are only two possible 
skew left braces $B$ with $(B,\cdot)\cong G$, given in Example~\ref{FirstExNLB} (a) and (b). 
Observe that in this case $(B,\star)\cong G$ 
(Carnahan and Childs \cite{CChilds}, Byott \cite{Byott2}).

\item There are some relations between the star and the multiplicative group: 
If $(B,\cdot)$ is a finite abelian group, then $(B,\star)$ is solvable. On the other hand, if $(B,\star)$
is a finite nilpotent group, then $(B,\cdot)$ is solvable (Byott \cite{Byott3}). 

\item There exist examples of finite skew left braces $B$ with $(B,\cdot)$ solvable and $(B,\star)$ simple 
(Byott \cite[Theorem 3]{Byott3}). 
But it is an open question to know whether $(B,\star)$ finite solvable implies that $(B,\cdot)$ is solvable. 
It is also open to know whether $(B,\cdot)$ finite nilpotent implies $(B,\star)$ is solvable. 

\end{enumerate}

\bibliographystyle{amsplain}
\bibliography{testbib}

\vspace{30pt}
 \noindent \begin{tabular}{llllllll}
 D. Bachiller  \\
 Departament de Matem\`atiques \\
 Universitat Aut\`onoma de Barcelona  \\
08193 Bellaterra (Barcelona), Spain  \\
 dbachiller@mat.uab.cat \\
\end{tabular}

\end{document}